\documentclass[11pt]{article}
\usepackage[margin=1in]{geometry}
\usepackage{float,longtable}
\usepackage{enumitem,mathtools,microtype,setspace}
\usepackage[dvipsnames]{xcolor}
\definecolor{BasisMassLink}{RGB}{22,67,114}
\definecolor{BasisMassCite}{RGB}{35,102,78}
\definecolor{BasisMassURL}{RGB}{0,91,150}
\usepackage[
  colorlinks=true,
  linkcolor=BasisMassLink,
  citecolor=BasisMassCite,
  urlcolor=BasisMassURL,
  linktoc=all
]{hyperref}
\makeatletter

\RequirePackage{amsmath}
\RequirePackage{amssymb}
\RequirePackage{amsthm}
\RequirePackage{bm}
\usepackage{bbm}
\RequirePackage{url}
\usepackage{natbib}
\usepackage{multirow}
\usepackage{graphicx}
\usepackage{subfigure}
\usepackage{makecell}
\usepackage{booktabs}
\usepackage{array}
\usepackage{algorithm}
\usepackage{algorithmic}
\usepackage{dsfont}
\RequirePackage[dvipsnames]{xcolor}
\RequirePackage[capitalise,noabbrev,nameinlink]{cleveref}

\newcommand{\tr}{\mathop{\mathrm{tr}}}

\newcommand{\diag}{{\rm diag}}

\newcommand*{\one}{{\bm 1}} 

\newtheoremstyle{mytheoremstyle} 
    {\topsep}                    
    {\topsep}                    
    {\normalfont}                   
    {}                           
    {\bfseries}                   
    {.}                          
    {.5em}                       
    {}  

\theoremstyle{mytheoremstyle}
\theoremstyle{plain}
\crefname{assumption}{Assumption}{Assumptions}
\crefname{theorem}{Theorem}{Theorems}
\Crefname{theorem}{Theorem}{Theorems}
\crefname{lemma}{Lemma}{Lemmas}
\Crefname{lemma}{Lemma}{Lemmas}
\crefname{proposition}{Proposition}{Propositions}
\Crefname{proposition}{Proposition}{Propositions}
\crefname{definition}{Definition}{Definitions}
\Crefname{definition}{Definition}{Definitions}
\crefname{corollary}{Corollary}{Corollaries}
\Crefname{corollary}{Corollary}{Corollaries}
\crefname{remark}{Remark}{Remarks}
\Crefname{remark}{Remark}{Remarks}
\crefname{example}{Example}{Examples}
\Crefname{example}{Example}{Examples}
\crefname{conjecture}{Conjecture}{Conjectures}
\Crefname{conjecture}{Conjecture}{Conjectures}
\crefname{fact}{Fact}{Facts}
\Crefname{fact}{Fact}{Facts}
\crefname{claim}{Claim}{Claims}
\Crefname{claim}{Claim}{Claims}
\Crefname{assumption}{Assumption}{Assumptions}
\crefname{condition}{Condition}{Conditions}
\Crefname{condition}{Condition}{Conditions}
\crefname{equation}{Equation}{Equations}
\Crefname{equation}{Equation}{Equations}
\crefname{table}{Table}{Tables}
\Crefname{table}{Table}{Tables}
\crefname{section}{Section}{Sections}
\Crefname{section}{Section}{Sections}
\crefname{subsection}{Section}{Sections}
\Crefname{subsection}{Section}{Sections}

\ifx\BlackBox\undefined
\newcommand{\BlackBox}{\rule{1.5ex}{1.5ex}}  
\fi

\ifx\QED\undefined
\def\QED{~\rule[-1pt]{5pt}{5pt}\par\medskip}
\fi

\ifx\proof\undefined
\newenvironment{proof}{\par\noindent{\bf Proof\ }}{\hfill\BlackBox\\[2mm]}
\fi
\ifx\solution\undefined
\newenvironment{solution}{\par\noindent{\bf Solution\ }}{\hfill\BlackBox\\[2mm]}
\fi

\ifx\theorem\undefined
\theoremstyle{plain}
\newtheorem{theorem}{Theorem}
\fi
\ifx\example\undefined
\theoremstyle{definition}
\newtheorem{example}[theorem]{Example}
\fi
\ifx\property\undefined
\theoremstyle{plain}
\newtheorem{property}{Property}
\fi
\ifx\lemma\undefined
\theoremstyle{plain}
\newtheorem{lemma}[theorem]{Lemma}
\fi
\ifx\proposition\undefined
\theoremstyle{plain}
\newtheorem{proposition}[theorem]{Proposition}
\fi
\ifx\remark\undefined
\theoremstyle{remark}
\newtheorem{remark}[theorem]{Remark}
\fi
\ifx\corollary\undefined
\theoremstyle{plain}
\newtheorem{corollary}[theorem]{Corollary}
\fi
\ifx\definition\undefined
\theoremstyle{definition}
\newtheorem{definition}[theorem]{Definition}
\fi
\ifx\conjecture\undefined
\theoremstyle{plain}
\newtheorem{conjecture}[theorem]{Conjecture}
\fi
\ifx\fact\undefined
\theoremstyle{plain}
\newtheorem{fact}[theorem]{Fact}
\fi
\ifx\claim\undefined
\theoremstyle{plain}
\newtheorem{claim}[theorem]{Claim}
\fi
\ifx\assumption\undefined
\theoremstyle{definition}
\newtheorem{assumption}[theorem]{Assumption}
\fi
\ifx\condition\undefined
\theoremstyle{definition}
\newtheorem{condition}[theorem]{Condition}
\fi

\let\orig@example\example \let\endorig@example\endexample
\renewenvironment{example}{\crefalias{theorem}{example}\orig@example}{\endorig@example}

\let\orig@lemma\lemma \let\endorig@lemma\endlemma
\renewenvironment{lemma}{\crefalias{theorem}{lemma}\orig@lemma}{\endorig@lemma}

\let\orig@proposition\proposition \let\endorig@proposition\endproposition
\renewenvironment{proposition}{\crefalias{theorem}{proposition}\orig@proposition}{\endorig@proposition}

\let\orig@remark\remark \let\endorig@remark\endremark
\renewenvironment{remark}{\crefalias{theorem}{remark}\orig@remark}{\endorig@remark}

\let\orig@corollary\corollary \let\endorig@corollary\endcorollary
\renewenvironment{corollary}{\crefalias{theorem}{corollary}\orig@corollary}{\endorig@corollary}

\let\orig@definition\definition \let\endorig@definition\enddefinition
\renewenvironment{definition}{\crefalias{theorem}{definition}\orig@definition}{\endorig@definition}

\let\orig@conjecture\conjecture \let\endorig@conjecture\endconjecture
\renewenvironment{conjecture}{\crefalias{theorem}{conjecture}\orig@conjecture}{\endorig@conjecture}

\let\orig@fact\fact \let\endorig@fact\endfact
\renewenvironment{fact}{\crefalias{theorem}{fact}\orig@fact}{\endorig@fact}

\let\orig@claim\claim \let\endorig@claim\endclaim
\renewenvironment{claim}{\crefalias{theorem}{claim}\orig@claim}{\endorig@claim}

\let\orig@assumption\assumption \let\endorig@assumption\endassumption
\renewenvironment{assumption}{\crefalias{theorem}{assumption}\orig@assumption}{\endorig@assumption}

\let\orig@condition\condition \let\endorig@condition\endcondition
\renewenvironment{condition}{\crefalias{theorem}{condition}\orig@condition}{\endorig@condition}

\numberwithin{equation}{section}
\numberwithin{theorem}{section}

\AddToHook{cmd/@thm/before}{\LinkTargetOff}
\AddToHook{cmd/deferred@thm@head/after}{%
  \LinkTargetOn
  \AddToHookNext{para/begin}{\MakeLinkTarget{\@currentcounter}}%
  \leavevmode\ignorespaces
}
\makeatother
\numberwithin{theorem}{section}

\crefname{equation}{Equation}{Equations}
\Crefname{equation}{Equation}{Equations}
\crefname{section}{Section}{Sections}
\Crefname{section}{Section}{Sections}
\crefname{subsection}{Section}{Sections}
\Crefname{subsection}{Section}{Sections}
\crefname{subsubsection}{Section}{Sections}
\Crefname{subsubsection}{Section}{Sections}
\crefname{table}{Table}{Tables}
\Crefname{table}{Table}{Tables}
\crefname{appendix}{Appendix}{Appendices}
\Crefname{appendix}{Appendix}{Appendices}

\newcommand{\R}{\mathbb R}

\hypersetup{
  pdftitle={A Sharp Basis-Mass Inequality for Binary Matroids without a Fano Minor},
  pdfauthor={Haozhou Zhang},
  pdfsubject={A sharp basis-generating-polynomial inequality for binary matroids without a Fano minor},
  pdfkeywords={binary matroid, regular matroid, Fano matroid,
               basis-generating polynomial, hypergraph Lagrangian,
               Seymour decomposition, Turan basis density}
}

\title{A Sharp Basis-Mass Inequality\\
for Binary Matroids without a Fano Minor}
\author{Haozhou Zhang\\
Department of Mathematics and Statistics\\
University of Idaho\\
\href{mailto:zhan3063@vandals.uidaho.edu}{\texttt{zhan3063@vandals.uidaho.edu}}}
\date{}
\makeatletter
\renewcommand{\@maketitle}{%
  \begin{center}
    {\LARGE\@title\par}\vskip .7em
    {\begin{tabular}{c}\@author\end{tabular}\par}
  \end{center}}
\makeatother

\begin{document}
\maketitle

\begin{abstract}
Let \(M\) be a binary matroid of rank \(r\) with no \(F_7\)-minor, and let
\(B_M^\star\) denote the maximum of its basis-generating polynomial over
nonnegative weights of total mass \(r\).  We prove the sharp inequality
\(B_M^\star\leq 2^r/(r+1)\), with equality holding if and only if the
simplification of \(M\) is isomorphic to \(M(K_{r+1})\).
Binary matroids with no \(F_7\)-minor also form the largest minor-closed
class of matroids on which this inequality holds.
For each fixed rank \(r\) and all sufficiently large \(n\), every
\(n\)-element rank-\(r\) matroid maximising the number of bases in this
class is graphic.
\end{abstract}



\section{Introduction}
\label{sec:introduction}

For a rank-\(r\) matroid \(M\) on a finite ground set \(E\), define its
basis-generating polynomial by
\[
 B_M(\lambda):=\sum_{B\in\mathcal B(M)}\prod_{e\in B}\lambda_e,
\]
where \(\mathcal B(M)\) is its set of bases.  Write \(B_M^\star\) for the
maximum of \(B_M\) on the weight simplex, the set of nonnegative weights
satisfying \(\sum_{e\in E}\lambda_e=r\).
For \(r\ge1\), rescaling to total weight one gives \(r^{-r}B_M^\star\),
the hypergraph Lagrangian of \((E,\mathcal B(M))\) in the normalisation of
\citet[\S4]{vanDerPolWalshWigal2025}.

A matroid is \emph{binary} if it is representable over \(\mathbb F_2\).
As shown by \citet[Corollary~2.4 and Section~2.4]{Alon2024} and
\citet[Theorem~4.4]{vanDerPolWalshWigal2025}, the projective geometry
\(\mathrm{PG}(r-1,2)\) with uniform weights attains the maximum over all
rank-\(r\) binary matroids.

A \emph{simplification} of \(M\) is obtained by deleting its loops and
retaining one representative of each parallel class \citep{Oxley2011}.

\begin{theorem}
\label{thm:introduction-basis-mass}
Every binary matroid \(M\) of rank \(r\geq1\) with no \(F_7\)-minor satisfies
\[
 B_M^\star\leq\frac{2^r}{r+1}.
\]
A feasible weighting \(\lambda\) satisfies
\(B_M(\lambda)=2^r/(r+1)\) if and only if every loop has weight zero, the
simplification of \(M\) is isomorphic to \(M(K_{r+1})\), and every
parallel class has total weight \(2/(r+1)\).
\end{theorem}

\paragraph{Proof outline.}
We first treat regular matroids, those representable over every field.
\citet[Theorem~14.3]{Seymour1980} proved that every such matroid can be obtained from
graphic matroids, cographic matroids, and \(R_{10}\) by direct sums,
two-sums, and three-sums.
After the optimisation and direct-sum lemmas in
\Cref{sec:basis-mass-framework,sec:regular-direct-sums}, we prove the graphic
and \(R_{10}\) bounds in \Cref{sec:regular-base-pieces}.  The cographic proof
uses edge deletions to lower the rank while controlling the loss of basis
mass (\Cref{sec:cographic-reduction}).

For two-sums, restoring the shared element gives a product bound.
For a proper regular three-sum
\(M=N_1\oplus_3N_2\) with factor ranks \(a,b\), the rank is
\(r=a+b-2\), and a product bound alone need not give
the sharp bound \(c_r:=2^r/(r+1)\).  The triangle-completion argument
supplies the stronger estimate
\[
 B_M^\star<\frac{r}{r+1}B_{N_1}^\star B_{N_2}^\star
 \qquad(a,b\ge3),
\]
whose extra factor is sufficient because
\(\frac{r}{r+1}c_ac_b\le c_r\).
\Cref{sec:regular-two-three-sums} proves these comparisons and handles
rank-two factors separately.  In \Cref{sec:seymour-assembly} we combine the
estimates to obtain the regular bound.  At equality, the positive support
meets \(\binom{r+1}{2}\) parallel classes, the most a rank-\(r\) regular
matroid can have.

\citet{Seymour1981Multicommodity} also showed that every binary matroid with no \(F_7\)-minor can be
obtained from regular matroids and copies of the dual Fano matroid \(F_7^*\)
by direct sums and two-sums (see also \citealp[Corollary~12.2.6]{Oxley2011}).
The sum inequalities
therefore extend the bound to the full class once we treat \(F_7^*\)
in \Cref{sec:fano-minor-free}.  At equality
we retain every parallel class and show that the simplification
is regular, so the regular equality classification applies.

\pagebreak[2]
\begin{samepage}
\begin{corollary}
\label{cor:largest-minor-closed-class}
The binary matroids with no \(F_7\)-minor form the largest minor-closed class
of finite matroids for which every member \(M\) satisfies
\[
 B_M^\star\leq\frac{2^{r(M)}}{r(M)+1}.
\]
\end{corollary}
\end{samepage}

For matroids in this class, \citet[Problem~7.7]{vanDerPolWalshWigal2025}
ask for the maximum number of bases at fixed rank \(r\) and ground-set
size \(n\).  When \(n\ge r\ge1\), uniform weights give
\(B_M(r/n,\ldots,r/n)=(r/n)^r|\mathcal B(M)|\), so
\Cref{thm:introduction-basis-mass} yields
\[
 |\mathcal B(M)|
 \leq\frac{2^r}{r+1}\left(\frac nr\right)^r.
\]
For connected graphs this recovers the bound of \citet{Kelmans1967} on the
number of spanning trees in terms of the numbers of vertices and edges.
\Cref{cor:fano-free-basis-count} determines the maximum when
\(\binom{r+1}{2}\) divides \(n\).  For each fixed rank,
\Cref{cor:fano-free-limiting-density} gives the limiting basis density
\(2^r r!/[r^r(r+1)]\), and \Cref{cor:fano-free-eventual-graphic} shows that
every maximiser is graphic for all sufficiently large \(n\).
For simple matroids in this class, combining the theorem with the size bound
of \citet[Corollary~6.4]{Kung1993Extremal} gives
\(|\mathcal B(M)|\leq(r+1)^{r-1}\), with equality if and only if
\(M\cong M(K_{r+1})\).

\section{Optimising the basis polynomial}
\label{sec:basis-mass-framework}

Since the proof proceeds by decomposition, each step requires us to compare
the optimum of a matroid with the optima of its pieces.  We use three facts in
making these comparisons.  First, a maximiser
may be restricted to its positive support, where the Karush--Kuhn--Tucker
(KKT) conditions take a simple form.  Second, a totally unimodular representation
expresses the basis polynomial as a weighted Gram determinant, so its
derivatives can be computed directly.  Third, a parallel class may be replaced
by one representative carrying its total weight.  Restricting to
positive support suffices for estimating the optimum, but the equality
classification will also require control of zero-weight nonloops.

\begin{lemma}
\label{lem:positive-support-reduction}
Let \(M\) have rank \(r\geq1\), let \(\lambda\) maximise \(B_M\) on the
weight simplex, and suppose
\(B_M(\lambda)>0\).  If \(A=\{e:\lambda_e>0\}\), then
\begin{enumerate}[label=\textup{(\roman*)}]
\item \(r(M|A)=r\);
\item \(B_M(\lambda)=B_{M|A}(\lambda_A)\);
\item \(\lambda_A\) is an interior maximiser for \(M|A\); and
\item for every \(e\in A\), \(\partial_e B_M(\lambda)=B_M(\lambda)\).
\end{enumerate}
For \(e\notin A\), the one-sided KKT inequality is
\(\partial_eB_M(\lambda)\le B_M(\lambda)\).
\end{lemma}

\begin{proof}
Since \(B_M(\lambda)>0\), some basis is contained in \(A\), and hence
\(r(M|A)=r\), proving (i).  Every monomial meeting \(E\setminus A\)
vanishes, so the restriction preserves the basis mass and gives (ii).  Any
improvement of the restricted weighting could be extended by zero to improve
the original one; thus \(\lambda_A\) is an interior maximiser for \(M|A\),
proving (iii).

At an interior maximiser, all partial derivatives have a common value \(\xi\).
Euler's identity and
\(\sum_{e\in A}\lambda_e=r\) give
\[
 rB_M=\sum_{e\in A}\lambda_e\partial_eB_M=\xi r,
\]
hence \(\xi=B_M\), which is (iv).  A zero coordinate permits motion only into
the simplex, and its KKT condition gives the final one-sided inequality.
\end{proof}

By Tutte's representation theorem \citep{Tutte1958a,Tutte1958b}, a matroid
is regular if and only if it admits a totally unimodular real representation
(see also
\citealp[Theorem~6.6.3 and Lemma~2.2.21]{Oxley2011}).  Such a representation
gives a determinantal formula for the basis polynomial, so the preceding
derivative condition can be read as a statement about the lengths of the
representing columns.  This is the
geometric form of optimality used later; in optimal-design terminology, it is
the Kiefer--Wolfowitz equivalence condition \citep{KieferWolfowitz1960}.

\begin{lemma}
\label{lem:regular-gram-calculus}
Let \(A=(a_e)_{e\in E}\) be a full-row-rank totally unimodular
representation of a rank-\(r\) regular matroid \(M\), and put
\[
 H(\lambda):=\sum_{e\in E}\lambda_ea_ea_e^{\top}.
\]
Then
\[
 B_M(\lambda)=\det H(\lambda).
\]
If \(H(\lambda)\) is positive definite, define
\(
 \langle u,v\rangle_H:=u^{\top}H(\lambda)^{-1}v
\).
For distinct \(e,f\),
\begin{align}
 \frac{\partial_eB_M(\lambda)}{B_M(\lambda)}
 &=\langle a_e,a_e\rangle_H,
 \label{eq:regular-gram-first-derivative}\\
 \frac{\partial_{ef}B_M(\lambda)}{B_M(\lambda)}
 &=\det\begin{pmatrix}
   \langle a_e,a_e\rangle_H&\langle a_e,a_f\rangle_H\\
   \langle a_f,a_e\rangle_H&\langle a_f,a_f\rangle_H
  \end{pmatrix}.
 \label{eq:regular-gram-second-derivative}
\end{align}
Consequently, after restricting an optimal weighting to its positive
support, every remaining column has \(H\)-norm one.
\end{lemma}

\begin{proof}
Cauchy--Binet expands the Gram determinant as
\[
 \det H(\lambda)
 =\sum_{|S|=r}\det(A_S)^2\prod_{e\in S}\lambda_e.
\]
Total unimodularity makes each squared maximal minor equal to one when \(S\)
is a basis of \(M\), and to zero otherwise.  The expansion therefore equals
\(B_M(\lambda)\).  Jacobi's formula gives
\eqref{eq:regular-gram-first-derivative}; a second application, equivalently
the rank-two determinant lemma, gives
\eqref{eq:regular-gram-second-derivative}.

On the positive support, \(H(\lambda)\) is positive definite and
\Cref{lem:positive-support-reduction} gives \(\partial_eB_M=B_M\) for every
\(e\) in that support.  Equation \eqref{eq:regular-gram-first-derivative} gives
\(\lVert a_e\rVert_H=1\) for every such \(e\).
\end{proof}

Parallel elements introduce no additional optimisation.  Replacing one
parallel element by another preserves a basis, so only their total weight
affects the basis polynomial.

\begin{lemma}
\label{lem:parallel-merging}
If \(e,f\) are distinct parallel elements, then \(B_M\) depends on their weights only
through \(\lambda_e+\lambda_f\).  Consequently deleting all but one element
of a parallel class and assigning its total weight to the remaining element
preserves the basis-polynomial value and the total weight.
\end{lemma}

\begin{proof}
No basis contains both \(e\) and \(f\), and replacing one by the other gives a
bijection between the two families of bases that contain exactly one of them.
Pairing the corresponding monomials factors out
\(\lambda_e+\lambda_f\), so neither the basis mass nor the total weight
changes when the pair is merged.
\end{proof}
\section{Direct sums}
\label{sec:regular-direct-sums}

Direct sums constitute the simplest of Seymour's sum operations.  Their basis
polynomials factor, and homogeneity determines how the available mass is
divided between the summands.  The argument works on the closed simplex, so
it also covers rank-zero summands.  We use this identity in the cographic
reductions and in the regular and \(F_7\)-minor-free cases.

\begin{lemma}
\label{lem:one-sum-scaling}
Let \(M=M_1\oplus M_2\), where \(M_i\) has rank \(r_i\geq0\).  Then
\[
 B_M^\star=B_{M_1}^\star B_{M_2}^\star.
\]
\par\goodbreak
In particular, a rank-zero summand does not change the optimum.  If
\(r_1,r_2>0\) and
\(B_{M_i}^\star\leq2^{r_i}/(r_i+1)\) for \(i=1,2\), then, with
\(r=r_1+r_2\),
\[
 B_M^\star
 \leq \frac{2^r}{(r_1+1)(r_2+1)}
 <\frac{2^r}{r+1}.
\]
\end{lemma}

\begin{proof}
The basis polynomial factors:
\[
 B_M(\lambda)=B_{M_1}(\lambda^{(1)})
                 B_{M_2}(\lambda^{(2)}).
\]
Suppose first that \(r_1,r_2>0\).  A maximiser for each factor, placed on the
two disjoint ground sets, gives a feasible weighting for \(M\).  Hence
\(B_M^\star\geq B_{M_1}^\star B_{M_2}^\star\).

The reverse inequality asks whether a different division of the total mass
can increase the product.  Let \(\lambda\) be feasible and put
\(s_i=\sum_{e\in E(M_i)}\lambda_e\), so \(s_1+s_2=r\).  If either \(s_1\)
or \(s_2\) is zero, then \(B_M(\lambda)=0\).  Otherwise homogeneity gives
\[
 B_{M_i}(\lambda^{(i)})
 =\left(\frac{s_i}{r_i}\right)^{r_i}
   B_{M_i}\!\left(\frac{r_i}{s_i}\lambda^{(i)}\right)
 \leq \left(\frac{s_i}{r_i}\right)^{r_i}B_{M_i}^\star.
\]
The product of the remaining scale factors cannot exceed one: the weighted
arithmetic--geometric mean inequality yields
\[
 \left[
 \left(\frac{s_1}{r_1}\right)^{r_1}
 \left(\frac{s_2}{r_2}\right)^{r_2}
 \right]^{1/r}
 \leq \frac{r_1}{r}\frac{s_1}{r_1}
     +\frac{r_2}{r}\frac{s_2}{r_2}=1.
\]
Thus \(B_M(\lambda)\leq B_{M_1}^\star B_{M_2}^\star\), proving the reverse
inequality.

If instead \(r_1=0<r_2\), then \(M_1\) consists only of loops and
\(B_{M_1}=1\).  For any feasible \(\lambda\), the mass \(s_2\) on the second
factor is at most \(r_2\).  If \(s_2=0\), the desired bound holds.  If
\(s_2>0\), homogeneity gives
\[
 B_M(\lambda)=B_{M_2}(\lambda^{(2)})
 \leq \left(\frac{s_2}{r_2}\right)^{r_2}B_{M_2}^\star
 \leq B_{M_2}^\star.
\]
Putting zero weight on the loop summand attains equality.  If both ranks are
zero, both sides of the asserted identity are \(1\).

\enlargethispage{\baselineskip}
When both ranks are positive, the assumed bounds give the strict inequality
in the lemma, since
\((r_1+1)(r_2+1)=r+1+r_1r_2>r+1\).
\end{proof}
\section{Graphic matroids and \texorpdfstring{\(R_{10}\)}{R10}}
\label{sec:regular-base-pieces}

We first prove the inequality for graphic matroids and for the ten-element
matroid \(R_{10}\).  In the graphic case, the weighted matrix-tree theorem
expresses the basis mass in terms of the Laplacian eigenvalues, whose product
is bounded by the arithmetic--geometric mean inequality.  For \(R_{10}\), we
use the signless incidence matrix of \(K_5\), together with Hadamard's inequality
and the arithmetic--geometric mean inequality.

\subsection{Graphic matroids}

\begin{theorem}
\label{prop:graphic-bound}
Let \(G\) be a connected multigraph on \(r+1\) vertices.  For nonnegative
edge weights satisfying \(\sum_e\lambda_e=r\),
\[
 B_{M(G)}(\lambda)=\sum_{T\text{ spanning tree}}\prod_{e\in T}\lambda_e
 \le\frac{2^r}{r+1}.
\]
Equality holds if and only if every loop has weight zero and the edges
joining each pair of distinct vertices have total weight \(2/(r+1)\).
\end{theorem}

\begin{proof}
If \(r=0\), both sides equal \(1\).  Assume \(r\geq1\), and let
\[
 L=\sum_{e=ij}\lambda_e(e_i-e_j)(e_i-e_j)^{\top}
\]
be the weighted Laplacian (loops contribute zero), and let
\(\nu_1,\ldots,\nu_r\) be its eigenvalues on \(\one^\perp\), zeros included.
The weighted matrix-tree theorem \citep{Chaiken1982} gives
\[
 B_{M(G)}(\lambda)=\frac1{r+1}\prod_{k=1}^r\nu_k.
\]
Moreover
\[
 \sum_{k=1}^r\nu_k=\tr L
 =2\sum_{e\text{ nonloop}}\lambda_e\le2r.
\]
Arithmetic--geometric mean therefore gives
\[
 B_{M(G)}(\lambda)
 \le\frac1{r+1}\left(\frac{2r}{r}\right)^r
 =\frac{2^r}{r+1}.
\]
If equality holds, the loops have total weight zero and
\(\nu_1=\cdots=\nu_r=2\).  Hence
\(L=2(I-\one\one^{\top}/(r+1))\).  Reading the off-diagonal entries
shows that the total weight between every vertex pair is \(2/(r+1)\).
Conversely, these weights attain equality.
\end{proof}

\subsection{The matroid \texorpdfstring{\(R_{10}\)}{R10}}
\begin{samepage}

We represent \(R_{10}\) by the signless incidence matrix of \(K_5\).  For the
associated weighted Gram matrix, the quadratic form in the all-ones direction
is fixed, so the determinant estimate reduces to its four-dimensional
orthogonal complement.  We first verify the representation and then carry
out this estimate.

\end{samepage}
\begin{proposition}
\label{prop:R10-bound}
For the rank-five regular matroid \(R_{10}\),
\[
 B_{R_{10}}^\star=\frac{81}{16}<\frac{16}{3}=\frac{2^5}{6}.
\]
The unique maximiser assigns weight \(1/2\) to every element.
\end{proposition}

\begin{proof}
Let \(U\) be the signless incidence matrix of \(K_5\), with columns
\[
 u_{ij}=e_i+e_j
 \qquad(1\leq i<j\leq5).
\]

To identify the represented matroid, consider a five-edge set \(F\).  The
kernel of \(U_F^{\top}\) consists of the
vectors satisfying \(x_i+x_j=0\) on every edge of \(F\).  Thus \(U_F\) is
nonsingular if and only if every component of \(([5],F)\) is nonbipartite.
Since five simple edges on five vertices cannot contain two nonbipartite
components, this means that \(F\) is connected and its unique cycle is odd.
Deleting leaves reduces \(U_F\) to the signless incidence matrix of that odd
cycle, and hence
\[
 |\det U_F|=2.
\]
Choose the five-cycle \(C=(12,23,34,45,15)\).  Ordering the remaining
columns as \((35,14,25,13,24)\), row reduction gives
\[
 U_C^{-1}U=
 \left[
 \begin{array}{ccccc|rrrrr}
 1&0&0&0&0&-1& 1& 0& 0& 1\\
 0&1&0&0&0& 1&-1& 1& 0& 0\\
 0&0&1&0&0& 0& 1&-1& 1& 0\\
 0&0&0&1&0& 0& 0& 1&-1& 1\\
 0&0&0&0&1& 1& 0& 0& 1&-1
 \end{array}
 \right].
\]
For this matrix, every maximal minor is
\(\det U_F/\det U_C\in\{0,\pm1\}\).  Every minor of the right-hand
block extends to a maximal minor, up to sign, by adjoining complementary identity
columns, so the displayed matrix is totally unimodular.  It is the representation
of \(R_{10}\) in \citet[Fig.~1, p.~328]{Seymour1980}, so \(U\) represents \(R_{10}\) over
\(\mathbb R\).  Let
\(
 G=U\diag(\lambda)U^{\top}
\).  Since every nonzero maximal minor of \(U\) equals \(\pm2\),
Cauchy--Binet gives
\begin{equation}
 \det G=4B_{R_{10}}(\lambda).
 \label{eq:R10-signless-determinant}
\end{equation}

It remains to bound this determinant under the mass constraint.  Suppose
\(\sum_{ij}\lambda_{ij}=5\), and put
\(v=\one/\sqrt5\).  Every column of \(U\) has squared norm \(2\) and
inner product \(2/\sqrt5\) with \(v\).  Therefore
\[
 \tr G=10,
 \qquad
 v^{\top}Gv=4.
\]
Thus the entry in the all-ones direction is fixed at \(4\), while the
remaining four diagonal entries have sum \(6\).  In an orthonormal basis whose
first vector is \(v\), Hadamard's inequality followed by arithmetic--geometric
mean gives
\[
 \det G
 \leq4\left(\frac{10-4}{4}\right)^4
 =\frac{81}{4}.
\]
By
\eqref{eq:R10-signless-determinant},
\[
 B_{R_{10}}(\lambda)\leq\frac{81}{16}.
\]
The uniform weighting \(\lambda_{ij}=1/2\) attains the bound, since
\[
 G=\frac32I+\frac12J.
\]
The eigenvalues of this matrix are \(4,3/2,3/2,3/2,3/2\), so equality holds.
Conversely, equality in Hadamard and arithmetic--geometric mean forces the
same matrix.  Its off-diagonal entries are the weights \(\lambda_{ij}\), which
proves uniqueness.  Finally, \(81/16<16/3\).
\end{proof}
\section{Cographic matroids}
\label{sec:cographic-reduction}

For cographic matroids, deleting edges while preserving connectivity lowers
the cycle rank.  The key is to choose a deletion whose proportional loss in
basis mass is no greater than the
corresponding loss in the normalising constant.  We first contract zero-weight
edges, replace parallel edges by single edges, and suppress degree-two
vertices.  Optimality excludes bridges, so it suffices to consider a simple
bridgeless graph with a maximiser of full support.  A single estimate with
inherited weights applies both to deleting one edge and to deleting two
edges incident with a cubic vertex.

\subsection{The cographic determinant and the planar case}

Let \(G=(V,E)\) be connected, with cycle rank \(r=|E|-|V|+1\).  For
nonnegative edge weights \(\lambda\), its cographic basis polynomial is
\[
 C_G(\lambda)
   :=\sum_{T\text{ spanning tree of }G}
       \prod_{e\notin T}\lambda_e.
\]
Write
\[
 c_r:=\frac{2^r}{r+1},
 \qquad
 C_G^\star:=\max_{\substack{\lambda\ge0\\\sum_{e\in E}\lambda_e=r}}C_G(\lambda),
 \qquad
 \rho(G):=\frac{C_G^\star}{c_r}.
\]

A basis of the bond matroid \(M^*(G)\) is the complement of a spanning tree, so the cotree
formula identifies \(C_G\) with \(B_{M^*(G)}\) and relates matroid minors
to graph deletion and contraction.  The next lemma expresses \(C_G\) as
a Gram determinant.  When \(C_G(\lambda)>0\), the inverse Gram matrix
defines an inner product.

\begin{lemma}
\label{lem:cographic-determinant}
Let \(G\) be connected, of cycle rank \(r\), and let \(Z\) be a totally
unimodular \(r\times |E(G)|\) matrix representing \(M^*(G)\).  Then
\[
 B_{M^*(G)}(\lambda)=\det\!\left(Z\diag(\lambda)Z^{\top}\right)
 =\sum_{T\text{ spanning tree}}\prod_{e\notin T}\lambda_e.
\]
\end{lemma}

\begin{proof}
The first identity follows from \Cref{lem:regular-gram-calculus}.  A basis of
\(M^*(G)\) is the complement of a spanning tree, so summing over spanning
trees gives the second identity.
\end{proof}

\begin{corollary}
\label{cor:planar-cographic-exit}
If \(G\) is planar, then
\[
 B_{M^*(G)}^\star\le\frac{2^r}{r+1}.
\]
\end{corollary}

\begin{proof}
Choose a planar dual \(G^*\).  Then \(M^*(G)\cong M(G^*)\), and
\Cref{prop:graphic-bound} gives the result.
\end{proof}

The planar case is now complete.  For the nonplanar case we shall use the
following lower bound on the cycle rank.

\begin{lemma}
\label{lem:nonplanar-cycle-rank}
Every connected nonplanar finite multigraph has cycle rank at least four.
\end{lemma}

\begin{proof}
Delete loops, merge parallel edges, remove vertices of degree at
most one, and suppress vertices of degree two, simplifying again after each
suppression.  Planarity is preserved in both directions, and cycle rank does
not increase.  The resulting graph is therefore nonplanar and simple, with
minimum degree at least three.  If it has \(n\) vertices, \(m\) edges, and
cycle rank \(s\), then
\[
 3n\leq2m=2(n+s-1).
\]
If \(s\leq3\), this gives \(n\leq4\), but every simple graph on at most four
vertices is planar.  Hence \(s\geq4\), and the original graph has cycle rank
at least four as well.
\end{proof}

\subsection{Contracting zero-weight edges}

A maximiser need not have full support.  If its value is positive, however,
every zero-weight edge belongs to each spanning tree that contributes a
nonzero monomial.  Contracting these edges therefore produces an interior
maximiser on the contracted graph.

\begin{proposition}
\label{prop:cographic-support-face}
Let \(\lambda\) be feasible and suppose \(C_G(\lambda)>0\).  Put
\(Z:=\{e:\lambda_e=0\}\) and \(S:=E\setminus Z\).  Then \(Z\) is a forest and
\[
 C_G(\lambda)=C_{G/Z}(\lambda|_S).
\]
The contraction preserves cycle rank.  Consequently, if \(\lambda\) is a
maximiser, then \(\lambda|_S\) is an interior maximiser for \(G/Z\).
\end{proposition}

\begin{proof}
The monomial indexed by a spanning tree \(T\) is nonzero only if \(T\)
contains every zero-weight edge.  Since \(C_G(\lambda)>0\), some spanning tree
contains \(Z\), and hence \(Z\) is a forest.  Moreover, the spanning trees of
\(G\) that contain \(Z\) correspond bijectively to the spanning trees of
\(G/Z\), with the same complementary edge sets in \(S\).  The polynomial
identity follows from this correspondence.  In matroid terms,
\(M^*(G/Z)=M^*(G)\setminus Z=M^*(G)|S\).

Contracting a forest removes the same number of vertices and edges, so it
does not change the cycle rank.  If \(\lambda\) is a maximiser, its restriction
must also maximise \(C_{G/Z}\) on its weight simplex: any improvement there could
be extended by zero on \(Z\), contradicting the choice of \(\lambda\).
\end{proof}

After the zero set is contracted, the maximiser has full support.  Three
further reductions produce smaller instances of the theorem.  An interior
maximiser cannot occur on a graph with a bridge.  Suppressing a degree-two
vertex preserves the optimum, and merging parallel edges of the graph lowers
its cycle rank.  Together with direct-sum reduction and planar duality,
these reductions leave a connected, simple, bridgeless, nonplanar graph of
minimum degree at least three.

\begin{lemma}
\label{lem:cographic-active-support-bridgeless}
Let \(H\) be a connected graph of cycle rank \(r\ge1\).  If \(C_H\) has
an interior maximiser \(\lambda\), then \(H\) has no bridge.
\end{lemma}

\begin{proof}
Suppose that \(H\) has a bridge.  It belongs to every spanning tree and to no
cographic basis, so \(C_H\) is independent of its weight and the corresponding
partial derivative is zero.  At an interior maximiser, however, homogeneity
and the total mass \(r\) make every partial derivative equal to the positive
Lagrange multiplier \(C_H(\lambda)\).  The two conclusions are incompatible.
\end{proof}

\begin{lemma}
\label{lem:cographic-parallel-class-merge}
Let \(e,f\) be parallel elements of a cographic matroid \(N\), and let \(N'\)
be obtained from \(N\setminus f\) by relabelling \(e\) as \(g\).  Then
\[
 B_N(\lambda)=B_{N'}(\mu),
 \qquad
 \mu_g=\lambda_e+\lambda_f,
 \quad \mu_x=\lambda_x\quad(x\ne g).
\]
Consequently \(B_N^\star=B_{N'}^\star\), with rank and normalised ratio
unchanged.  In graph language, suppressing a degree-two vertex in a
bridgeless graph, when its two incident edges are distinct, is such a merge.
\end{lemma}

\begin{proof}
A basis contains at most one member of a parallel class.  Pair each monomial
using \(e\) with the corresponding monomial using \(f\).  Their sum depends
on the two weights only through \(\lambda_e+\lambda_f=\mu_g\), while the
monomials involving neither \(e\) nor \(f\) are unchanged.  The
transformation preserves total weight and rank, and conversely the weight
of \(g\) may be split arbitrarily between \(e\) and \(f\).  The two maxima
are therefore equal.

At a degree-two vertex, the two incident edges form a two-edge bond.  A graph
bond is a circuit of its cographic matroid, so the two elements are parallel.
Suppressing the vertex performs the stated merge.
\end{proof}

\begin{lemma}
\label{lem:cographic-graph-parallel-descent}
Let \(H\) be a connected graph of cycle rank \(r\ge2\) with parallel edges
\(e,f\), and let \(H'\) be obtained from \(H-f\) by relabelling \(e\) as
\(g\).  Then \(H'\) has cycle rank \(r-1\).  When \(w:=\lambda_e+\lambda_f>0\),
\[
 C_H(\lambda)
 =(\lambda_e+\lambda_f)C_{H'}(\mu),
 \qquad
 \mu_g=\frac{\lambda_e\lambda_f}{\lambda_e+\lambda_f},
 \quad \mu_x=\lambda_x\quad(x\ne g).
\]
Consequently
\[
 C_H^\star\le\frac43C_{H'}^\star,
 \qquad
 \rho(H)\le\rho(H').
\]
\end{lemma}

\begin{proof}
The polynomial identity comes from the two possible ways in which a spanning
tree can meet the parallel pair.  No spanning tree contains both edges.  If
it contains exactly one, the other lies in the cotree and the two choices
together contribute \(\lambda_e+\lambda_f\).  If it contains neither, both
edges lie in the cotree and contribute \(\lambda_e\lambda_f\).  Matching these
two classes with the corresponding spanning trees of \(H'\) proves the
identity.

If \(w=\lambda_e+\lambda_f=0\), then every cographic basis has zero
weight, because a spanning tree cannot contain both parallel edges.  Thus
\(C_H(\lambda)=0\), so the claimed inequality holds.  We may assume \(w>0\).
The transformed vector then has mass
\[
 \sum_x\mu_x
 =r-w+\frac{\lambda_e\lambda_f}{w}
 \le r-\frac{3w}{4}.
\]
The only remaining loss is the rescaling of this mass to \(r-1\).
Homogeneity in degree \(r-1\) gives
\[
 C_H(\lambda)\le
 w\left(\frac{r-3w/4}{r-1}\right)^{r-1}C_{H'}^\star.
\]
To compare the resulting factor with the normalisation, put \(x=3w/4\).
The transformed mass is nonnegative, so \(0\le x\le r\).  Weighted
arithmetic--geometric mean gives
\[
 x^{1/r}
 \left(\frac{r-x}{r-1}\right)^{(r-1)/r}
 \le \frac{x+(r-x)}r=1.
\]
Consequently
\[
 w\left(\frac{r-3w/4}{r-1}\right)^{r-1}
 =\frac43x\left(\frac{r-x}{r-1}\right)^{r-1}
 \le\frac43.
\]
The cycle rank has fallen by one, and
\(c_r/c_{r-1}=2r/(r+1)\geq4/3\) for \(r\geq2\).  Thus the factor \(4/3\) in
the basis-mass estimate is no greater than the change in the normalising
constant, so \(\rho(H)\leq\rho(H')\).
\end{proof}

\subsection{A cubic star at an interior maximiser}

The one-edge estimate is strong when an individual edge carries enough
weight.  If all edge weights are small, we need a second move.  A cubic
vertex provides one, because its three incident edges form a cographic
circuit.  At an interior maximiser, the circuit relation implies that, after
reorientation, the three pairwise inner products of the corresponding columns
are all \(-1/2\).  This determines the mixed derivative for every pair.

\begin{lemma}
\label{lem:cographic-equilateral-star}
Let \(G\) be connected, simple, and bridgeless, let \(\lambda\) be an interior
maximiser of \(C_G\), and let \(e,f,g\) be the three edges incident with a
degree-three vertex \(v\).  For any two distinct members \(h,k\) of this star,
\begin{equation}
 \frac{\partial_{hk}C_G(\lambda)}{C_G(\lambda)}=\frac34.
 \label{eq:cographic-equilateral-star}
\end{equation}
\end{lemma}

\begin{proof}
The three edges must first be seen as a cographic circuit.  Every component
of \(G-v\) meets \(v\) in at least two edges, since a unique joining edge
would be a bridge.  As \(v\) has degree three, there can be only one such
component.  Thus \(G-v\) is connected, so \(\{e,f,g\}\) is a bond and
therefore a circuit of \(M^*(G)\).

Choose a full-row-rank totally unimodular representation \(Z=(z_x)\) of \(M^*(G)\).
Extend two of \(z_e,z_f,z_g\) to a column basis.  Cramer's rule expresses
the primitive circuit coefficients as ratios of nonzero maximal minors, so
total unimodularity makes each coefficient \(\pm1\).  After reorienting the
three columns, we may write their relation as
\[
 z_e+z_f+z_g=0.
\]
Put \(H=Z\diag(\lambda)Z^{\top}\).  By
\Cref{lem:regular-gram-calculus}, all positive-weight columns have unit length in the
inner product \(\langle u,w\rangle_H=u^{\top}H^{-1}w\).  Evaluating the
circuit relation in this inner product gives
\[
 1=\lVert z_e+z_f\rVert_H^2
  =2+2\langle z_e,z_f\rangle_H.
\]
Thus \(\langle z_e,z_f\rangle_H=-1/2\), and the same calculation applies to
the other two pairs.  The mixed-derivative identity from
\Cref{lem:regular-gram-calculus} now gives
\[
 \frac{\partial_{ef}C_G}{C_G}
 =\det\begin{pmatrix}1&-1/2\\-1/2&1\end{pmatrix}
 =\frac34.
\]
The other two pairs give the same value.  In particular, deleting any two
edges of the star retains exactly three quarters of the value at the original
weighting.
\end{proof}

\subsection{Edge deletion with inherited weights}

Both deletion moves use the same comparison.  A mixed derivative gives the
value retained at the original weights.  Any edges that become bridges may
then be assigned weight zero before the remaining mass is rescaled to the
new cycle rank.

\begin{lemma}
\label{lem:cographic-frozen-weight-deletion}
Let \(G=(V,E)\) be connected of cycle rank \(r\), and let \(\lambda\) be an
interior maximiser of \(C_G\) on \(\sum_{e\in E}\lambda_e=r\).  Let
\(S\subseteq D\subseteq E\), with \(|S|=k\) and \(1\le k\le r-1\), and
suppose every edge of \(D\setminus S\) is a bridge of \(G-S\).  Write
\(\partial_S:=\prod_{e\in S}\partial_e\), and put
\[
 \theta_S:=\frac{\partial_S C_G(\lambda)}{C_G(\lambda)},
 \qquad
 m_D:=\sum_{e\in D}\lambda_e.
\]
If \(\theta_S>0\), then \(G-S\) is connected of cycle rank \(r-k\),
\(m_D<r\), and
\begin{equation}
 \frac{\rho(G-S)}{\rho(G)}
 \ge
 \theta_S\frac{2^k(r-k+1)}{r+1}
 \left(\frac{r-k}{r-m_D}\right)^{r-k}.
 \label{eq:cographic-frozen-weight-deletion}
\end{equation}
\end{lemma}

\begin{proof}
The positive mixed derivative means that some cographic basis contains
\(S\), so its complementary spanning tree avoids \(S\).  Thus \(G-S\) is
connected and has cycle rank \(r-k\).  Graph deletion corresponds to matroid
contraction, \(M^*(G-S)=M^*(G)/S\).  Multiaffinity and the spanning-tree
correspondence give
\[
 C_{G-S}(\lambda|_{E\setminus S})
 =\partial_S C_G(\lambda)
 =\theta_S C_G(\lambda).
\]
Every edge of \(D\setminus S\) belongs to every spanning tree of \(G-S\),
so this polynomial is independent of those weights.  A cotree of \(G-S\)
has \(r-k\ge1\) edges, all outside \(D\).  Positivity of the original
weights therefore gives \(r-m_D=\sum_{e\notin D}\lambda_e>0\).

Define a weighting of \(G-S\) by
\[
 \widehat\lambda_e=
 \begin{cases}
  0,&e\in D\setminus S,\\[2pt]
  \displaystyle\frac{r-k}{r-m_D}\lambda_e,&e\in E\setminus D.
 \end{cases}
\]
Its total mass is \(r-k\).  Independence of the bridge weights and
homogeneity in degree \(r-k\) give
\[
 C_{G-S}^\star
 \ge C_{G-S}(\widehat\lambda)
 =\theta_S\left(\frac{r-k}{r-m_D}\right)^{r-k}C_G^\star.
\]
Finally,
\[
 \frac{\rho(G-S)}{\rho(G)}
 =\frac{C_{G-S}^\star}{C_G^\star}\frac{c_r}{c_{r-k}},
 \qquad
 \frac{c_r}{c_{r-k}}=\frac{2^k(r-k+1)}{r+1},
\]
which proves \eqref{eq:cographic-frozen-weight-deletion}.
\end{proof}

For one-edge deletion, first-order optimality makes the retained-value
factor equal to one.  Keeping all other weights before rescaling gives the
following bound in terms of the deleted edge alone.

\begin{proposition}
\label{prop:cographic-one-edge-descent}
Let \(G\) be connected of cycle rank \(r\ge2\), and let \(\lambda\) be an
interior maximiser of \(C_G\).  Then, for every edge \(e\), the graph \(G-e\)
is connected and
\begin{equation}
 \frac{\rho(G-e)}{\rho(G)}
 \ge D_r(\lambda_e):=
 \frac{2r}{r+1}
 \left(\frac{r-1}{r-\lambda_e}\right)^{r-1}.
 \label{eq:frozen-factor}
\end{equation}
\end{proposition}

\begin{proof}
By \Cref{lem:positive-support-reduction},
\(\partial_eC_G(\lambda)=C_G(\lambda)\), so \(\theta_{\{e\}}=1\).
Apply \Cref{lem:cographic-frozen-weight-deletion} with
\(k=1\) and \(S=D=\{e\}\).
\end{proof}

For two-edge deletion at a cubic star, the retained-value factor is
\(3/4\), but the third star edge then
becomes a bridge.  Its weight may also be discarded before rescaling.  The
rescaling cost therefore depends on the full mass of the star, rather than
only on the two deleted weights.

\begin{proposition}
\label{prop:cographic-conditioned-star-descent}
Let \(G\) be connected, simple, and bridgeless, of cycle rank \(r\ge3\), and let
\(\lambda\) be an interior maximiser of \(C_G\)
subject to \(\sum_e\lambda_e=r\).  Suppose that \(v\) has degree three, with
incident edges \(e,f,g\), and put
\[
 S:=\{e,f\},
 \qquad
 s_v:=\lambda_e+\lambda_f+\lambda_g.
\]
Then \(G-S\) is connected, has cycle rank \(r-2\), and
\begin{equation}
 \frac{\rho(G-S)}{\rho(G)}
 \ge
 \frac{3(r-1)}{r+1}
 \left(\frac{r-2}{r-s_v}\right)^{r-2}.
 \label{eq:conditioned-star-factor}
\end{equation}
The same assertion holds for any choice of two of the three edges incident
with \(v\).
\end{proposition}

\begin{proof}
By \Cref{lem:cographic-equilateral-star}, \(\theta_{\{e,f\}}=3/4\).
Deleting \(e,f\) makes \(v\) a leaf, so \(g\) is a bridge of \(G-S\).
Apply \Cref{lem:cographic-frozen-weight-deletion} with \(k=2\) and
\(D=\{e,f,g\}\).  Here \(m_D=s_v\), and
\((3/4)\,2^2(r-1)/(r+1)=3(r-1)/(r+1)\), giving
\eqref{eq:conditioned-star-factor} for any choice of the pair.
\end{proof}

\subsection{Finding an improving deletion}
\label{sec:all-rank-cographic}

Let \(G\) be connected, simple, bridgeless, and nonplanar, with minimum
degree at least three and cycle rank \(r\ge4\).  At edge weight
\(\lambda_e=1/2\), the one-edge bound \eqref{eq:frozen-factor} gives the
factor \(d_r\) below.  At star mass \(s_v=r/(r-1)\), the two-edge bound
\eqref{eq:conditioned-star-factor} gives \(F_r\).  We first show that both
factors exceed one, then prove by degree counting that at least one of these
thresholds is met.

\begin{lemma}
\label{lem:cographic-uniform-factor}
For integers \(r\ge4\), put
\[
 d_r:=\frac{2r}{r+1}
       \left(\frac{2(r-1)}{2r-1}\right)^{r-1},
 \qquad
 F_r:=\frac{3(r-1)}{r+1}
       \left(\frac{r-1}{r}\right)^{r-2}.
\]
Both sequences are strictly increasing.  Consequently
\begin{equation}
 d_r\ge d_4=\frac{1728}{1715}>1,
 \qquad
 F_r\ge F_4=\frac{81}{80}>1.
 \label{eq:uniform-two-edge-factor}
\end{equation}
\end{lemma}

\begin{proof}
It is enough to compare consecutive terms, and Bernoulli's inequality gives
the needed comparison in both cases.  For \(d_r\),
\[
 \frac{d_{r+1}}{d_r}
 =\frac{(r+1)^2}{r(r+2)}\frac{2r}{2r+1}
  \left(1+\frac1{(2r+1)(r-1)}\right)^{r-1}.
\]
Bernoulli gives
\[
 \left(1+\frac1{(2r+1)(r-1)}\right)^{r-1}
 \ge1+\frac1{2r+1}.
\]
It follows that
\[
 \frac{d_{r+1}}{d_r}
 \ge\frac{4(r+1)^3}{(r+2)(2r+1)^2}>1,
\]
since
\(4(r+1)^3-(r+2)(2r+1)^2=3r+2>0\).

The second sequence has a similar comparison.  Put \(t=1/r\); cancelling
consecutive factors gives
\[
 \frac{F_{r+1}}{F_r}
 =\frac1{(1+t-2t^2)(1-t^2)^{r-2}}.
\]
Bernoulli's inequality gives
\[
 (1-t^2)^{-(r-2)}
 =\left(1+\frac{t^2}{1-t^2}\right)^{r-2}
 \ge1+\frac{(r-2)t^2}{1-t^2}
 >1+(r-2)t^2=1+t-2t^2.
\]
Thus \(F_{r+1}/F_r>1\).  Both factors are therefore smallest at \(r=4\),
where they have the stated values.
\end{proof}

\begin{theorem}
\label{thm:cographic-unified-descent}
Let \(G\) be connected, simple, and bridgeless, with minimum degree at least
three and cycle rank \(r\ge4\).  Let
\(\lambda\) be an interior maximiser of \(C_G\) on
\(\sum_e\lambda_e=r\).  Then one of the following alternatives holds:
\begin{enumerate}
 \item some edge \(e\) satisfies
 \[
  G-e\text{ is connected},
  \qquad
  \frac{\rho(G-e)}{\rho(G)}>1;
 \]
 \item some cubic vertex \(v\) and two incident edges \(e,f\) satisfy
 \[
  G-\{e,f\}\text{ is connected},
  \qquad
  \frac{\rho(G-\{e,f\})}{\rho(G)}
  \ge F_r\ge\frac{81}{80}>1.
 \]
\end{enumerate}
The graph obtained by deletion has cycle rank \(r-1\) in the first case and
\(r-2\) in the second.
\end{theorem}

\begin{proof}
Suppose that \(\lambda_e\ge1/2\) for some edge \(e\).  The one-edge
deletion bound, \Cref{prop:cographic-one-edge-descent}, gives
\[
 \frac{\rho(G-e)}{\rho(G)}
 \ge
 D_r(\lambda_e):=
 \frac{2r}{r+1}
 \left(\frac{r-1}{r-\lambda_e}\right)^{r-1}.
\]
The function \(D_r(a)\) is strictly increasing for \(a<r\).  Hence
\Cref{lem:cographic-uniform-factor} gives
\[
 D_r(\lambda_e)\ge D_r(1/2)=d_r
 \ge\frac{1728}{1715}>1.
\]
Since \(G\) is bridgeless, \(G-e\) is connected.  Thus a heavy edge gives the
first alternative.

Suppose now that \(\lambda_e<1/2\) for every edge, and write
\(s_v=\sum_{e\ni v}\lambda_e\) for the mass incident with \(v\).
Counting each edge at both endpoints gives \(\sum_v s_v=2r\) and
\(\sum_v(\deg v-2)=2|E(G)|-2|V(G)|=2(r-1)\).
Since every vertex has degree at least three, the ratios
\(s_v/(\deg v-2)\) have weighted mean
\[
 \frac{\sum_v s_v}{\sum_v(\deg v-2)}=\frac{r}{r-1}>1
\]
with weights \(\deg v-2\).  Some vertex has ratio at least this mean.
Every vertex of degree \(d\ge4\) has \(s_v<d/2\le d-2\), so the chosen
vertex has degree three and satisfies \(s_v\ge r/(r-1)\).

Choose any two edges \(e,f\) incident with this vertex.  The cubic-star
deletion bound, \Cref{prop:cographic-conditioned-star-descent}, shows that
\(G-\{e,f\}\) is connected and that
\[
 \frac{\rho(G-\{e,f\})}{\rho(G)}
 \ge
 \frac{3(r-1)}{r+1}
 \left(\frac{r-2}{r-s_v}\right)^{r-2}.
\]
Here \(s_v<3/2<r\), so all denominators are positive.  Moreover,
\[
 r-s_v
 \le r-\frac r{r-1}
 =\frac{r(r-2)}{r-1},
\]
and hence
\[
 \frac{\rho(G-\{e,f\})}{\rho(G)}
 \ge
 \frac{3(r-1)}{r+1}
 \left(\frac{r-1}{r}\right)^{r-2}
 =F_r.
\]
By \Cref{lem:cographic-uniform-factor}, \(F_r\ge81/80>1\).  Thus the absence
of a heavy edge forces the second alternative.
\end{proof}

\subsection{Proof of the cographic theorem}

\begin{theorem}
\label{thm:analytic-cographic-bound}
For every connected graph \(G\) of cycle rank \(r\),
\[
 C_G^\star\le\frac{2^r}{r+1}.
\]
\end{theorem}

\begin{proof}
We induct lexicographically on \((r,|E(G)|)\).  If \(r=0\), then
\(C_G^\star=1=c_0\).  The first question is whether a maximiser lies on the boundary.
Choose a maximiser \(\lambda\), and contract its zero-weight forest using
\Cref{prop:cographic-support-face}.  The resulting graph \(H\) has the same
cycle rank.  It also has the same optimum, since the restricted maximiser gives
\(C_H^\star\ge C_G^\star\), while extension by zero gives the reverse
inequality.  If the zero set is nonempty, then \(|E(H)|<|E(G)|\), so the
induction on the number of edges applies.  We may henceforth assume that
\(\lambda\) has full support on \(G\).

Full support rules out bridges by
\Cref{lem:cographic-active-support-bridgeless}.  If \(G\) is planar, its
cographic matroid is graphic by planar duality, and
\Cref{cor:planar-cographic-exit} proves the claim.  Thus only the nonplanar
case remains, and \Cref{lem:nonplanar-cycle-rank} gives \(r\ge4\).

Each loop of \(G\) is a coloop of \(M^*(G)\) and splits off as a rank-one
summand.  A decomposition of \(G\) at a cut vertex gives a direct sum of
the corresponding bond matroids.  Since \(G\) is bridgeless, each part
contains a cycle and has positive cycle rank.  In both cases
\Cref{lem:one-sum-scaling} and rank induction apply.  If \(G\) has a pair
of parallel edges, let \(G'\) be obtained by deleting one of them.  Then
\Cref{lem:cographic-graph-parallel-descent} and rank induction give
\(\rho(G)\le\rho(G')\le1\).  A degree-two vertex may instead be suppressed.
By \Cref{lem:cographic-parallel-class-merge}, the resulting graph \(G'\) has
the same cycle rank and normalised optimum but fewer edges, so lexicographic
induction applies again.

If none of these reductions applies, the graph is connected, simple, and
bridgeless, and it has minimum degree at least three; degree-one vertices are
excluded by bridgelessness.  The chosen weighting remains an interior maximiser.
Thus the remaining graph satisfies the hypotheses of
\Cref{thm:cographic-unified-descent}.  In the first alternative, rank
induction gives
\[
 \rho(G)<\rho(G-e)\le1.
\]
In the second alternative, the same argument gives
\[
 \rho(G)<\rho(G-\{e,f\})\le1.
\]
In both cases the normalised ratio strictly increases on passing to a connected
graph of smaller cycle rank, \(r-1\) or \(r-2\) respectively.  This completes
the induction.
\end{proof}
\section{Two- and three-sums}
\label{sec:regular-two-three-sums}

We have now proved the bound for each of the three building blocks, and
\Cref{sec:regular-direct-sums} treated direct sums.  To complete the
induction, we prove the corresponding estimates for two-sums and three-sums.
For two-sums, a product comparison already preserves the rank-dependent
constant.  Three-sums lose two units of rank and require an additional
quantitative saving.

\subsection{Two-sums}

A two-sum joins its factors along an element \(z\) and then removes \(z\).
The basis on one side must use this common element and the basis on the other
side must avoid it.  To keep these two possibilities visible, let \(N_i\) be
a matroid of rank \(r_i\), where \(E(N_1)\cap E(N_2)=\{z\}\) and \(z\) is
neither a loop nor a coloop of either factor.  We call the sum proper when
each factor has at least three elements.  In the variables outside \(z\), write
\[
 a_i=B_{N_i/z},\qquad c_i=B_{N_i\setminus z}.
\]
A basis of \(M=N_1\oplus_2N_2\) is uniquely
\((B_1\cup B_2)\setminus\{z\}\), where
\(B_i\in\mathcal B(N_i)\) and \(z\) belongs to exactly one of
\(B_1,B_2\).  The two choices are in bijection with
\(\mathcal B(N_1/z)\times\mathcal B(N_2\setminus z)\) and
\(\mathcal B(N_1\setminus z)\times\mathcal B(N_2/z)\), respectively.
Thus the two-sum \(M=N_1\oplus_2N_2\) has rank \(r_1+r_2-1\), and its basis
polynomial is
\begin{equation}
 B_M=a_1c_2+c_1a_2.
 \label{eq:two-sum-polynomial}
\end{equation}

\begin{lemma}
\label{lem:two-sum-submultiplicativity}
For every two-sum as above,
\[
 B_{N_1\oplus_2N_2}^\star
 \leq B_{N_1}^\star B_{N_2}^\star.
\]
\end{lemma}

\begin{proof}
Let \(\lambda\) maximise \(B_M\), and evaluate \(a_i,c_i\) at the
corresponding external weights.  Put
\[
 U=a_1c_2,\qquad V=c_1a_2,
 \qquad B=U+V=B_M(\lambda).
\]
The two-sum has a feasible weighting of positive basis mass, so maximality
gives \(B>0\).  Thus
\[
 \alpha=\frac UB,\qquad \beta=\frac VB
\]
are well defined, nonnegative, and sum to \(1\), even when one of \(U,V\)
vanishes.

The aim is to complete the external weights on each factor by assigning some
weight to \(z\).  The correct two weights are forced by the proportion in
which the two terms contribute to \(B\).  Let \(s_i\) be the total
external weight on side \(i\).  Euler's identity, applied according to the
degree on that side, gives
\[
 \sum_{e\in E(N_1)\setminus z}\lambda_e\partial_eB
   =r_1B-U.
\]
Only positive coordinates contribute to the left side.  For each such \(e\),
\Cref{lem:positive-support-reduction} gives \(\partial_e B_M(\lambda)=B\).
The same calculation on side \(2\) gives
\[
 s_1=r_1-\alpha,
 \qquad s_2=r_2-\beta.
\]
Consequently, giving \(z\) weight \(\alpha\) in \(N_1\) and weight \(\beta\)
in \(N_2\) places the two completed vectors on their weight simplices.  Hence
\[
 B_{N_1}^\star\geq \alpha a_1+c_1,
 \qquad
 B_{N_2}^\star\geq \beta a_2+c_2.
\]
It remains to compare the product of these two completions with the original
two-sum value.  Their difference factors as
\[
 (\alpha a_1+c_1)(\beta a_2+c_2)-B
 =\frac{c_1c_2(B-a_1a_2)^2}{B^2}\geq0.
\]
The identity remains valid when one of \(U,V\) vanishes, and the KKT argument involved only
positive coordinates.  Thus boundary maximisers require no separate case.
\end{proof}

\begin{corollary}
\label{cor:two-sum-propagation}
If \(r_1,r_2\geq1\) and each \(N_i\) satisfies
\(B_{N_i}^\star\leq2^{r_i}/(r_i+1)\), then, for
\(r=r_1+r_2-1\),
\[
 B_{N_1\oplus_2N_2}^\star\leq\frac{2^r}{r+1}.
\]
The inequality is strict if \(r_1,r_2\geq2\).
\end{corollary}

\begin{proof}
By \Cref{lem:two-sum-submultiplicativity}, it remains to compare the
constants.  With \(r=r_1+r_2-1\), the claimed inequality follows from
\((r_1+1)(r_2+1)-2(r_1+r_2)=(r_1-1)(r_2-1)\geq0\), and it is strict when
\(r_1,r_2\geq2\).
\end{proof}
\subsection{Compatible three-sum representation}
\label{sec:regular-three-sum}

The two-sum product comparison is not enough for three-sums.  If the factor
ranks are \(a,b\), the resulting rank is \(r=a+b-2\), so even a bound by
\(c_ac_b\) can exceed \(c_r\): for \(a=b=3\),
\(c_ac_b=4>c_4=16/5\).  For proper regular three-sums with \(a,b\geq3\),
we will prove the stronger comparison
\(B_M^\star<\frac{r}{r+1}B_{N_1}^\star B_{N_2}^\star\).
The saving \(r/(r+1)\) suffices because
\(\frac{r}{r+1}c_ac_b\leq c_r\).  A rank-two factor will be treated separately.

To obtain this saving, we must quantify how much basis mass is recovered
when the deleted triangle is restored to the summands.  We first choose
compatible coordinates for the common triangle.  Contraction of that
triangle and the ways external sets extend to bases determine the factors
in the basis-polynomial comparison.

Let \(N_1,N_2\) be regular matroids with
\(E(N_1)\cap E(N_2)=T=\{t_1,t_2,t_3\}\), where \(T\) is a coindependent
triangle in each summand, so \(T\) contains no cocircuit of either.
Write \(M=N_1\oplus_3N_2\) and \(r_i=r(N_i)\).  Under the convention of
\citet[\S9.3]{Oxley2011}, a three-sum requires at least seven elements in
each summand.  We call this case proper and also use \(\oplus_3\) for
smaller summands.  Put
\[
 v_1=(1,0)^{\top},\qquad
 v_2=(0,1)^{\top},\qquad
 v_3=(1,1)^{\top},
 \qquad V=(v_1\ v_2\ v_3).
\]

\begin{lemma}
\label{lem:compatible-three-sum-representation}
After basis pivots, row permutations, and row and column sign changes,
the two summands admit full-row-rank totally unimodular representations
\[
 A_i=
 \begin{pmatrix}V&P_i\\0&Q_i\end{pmatrix},
 \qquad
 i=1,2,
\]
in which the columns of \(V\) are indexed by \(t_1,t_2,t_3\).  Moreover,
\[
 A_M=
 \begin{pmatrix}
 P_1&P_2\\
 Q_1&0\\
 0&Q_2
 \end{pmatrix}
\]
is a full-row-rank totally unimodular representation of
\(M=N_1\oplus_3N_2\).  Here \(Q_i\) has \(r_i-2\) rows.
\end{lemma}

\begin{proof}
The two copies of \(T\) must agree over the integers.  For each \(i\), choose
a full-row-rank totally unimodular representation \(\widetilde A_i\).
Extend \(\{t_1,t_2\}\) to a basis and let \(U_i\) be the square submatrix
formed by its columns.  Then \(\det U_i=\pm1\), so
\(U_i^{-1}\widetilde A_i\) is an integral representation whose chosen basis
columns form the identity.  We check that this change of basis preserves
total unimodularity.  Its maximal minors differ from those of
\(\widetilde A_i\) only by the factor \(\det(U_i^{-1})=\pm1\).  For a smaller minor,
either one of its chosen identity columns is supported outside the chosen
rows, in which case the minor is zero, or we may adjoin the remaining
identity columns and obtain a maximal minor up to sign.  Thus the new
representation is again totally unimodular.

In these basis coordinates, the first two columns indexed by \(T\) are
\(e_1,e_2\).  Since \(T\) is a circuit, the third column lies in their span
and both coefficients are nonzero.  Total unimodularity makes these
coefficients \(\pm1\).  After row and column sign changes, both summands
therefore have the same triangle block
\[
 V=(e_1\ e_2\ e_1+e_2).
\]

The external columns must still retain the full rank of each summand.
Coindependence gives \(r(N_i\setminus T)=r_i\), since \(T\) contains no
cocircuit.  Hence
\(E_i:=\binom{P_i}{Q_i}\) has full row rank.

With the common triangle block fixed, reduction modulo two identifies the
matroid represented by \(A_M\).  For
\(X_i\subseteq E(N_i)\setminus T\), let \(\chi_i\) be its incidence vector.
Then
\[
 A_M\binom{\chi_1}{\chi_2}=0
 \quad\Longleftrightarrow\quad
 Q_i\chi_i=0\ (i=1,2),\qquad P_1\chi_1=P_2\chi_2.
\]
Because \(V:\mathbb F_2^3\to\mathbb F_2^2\) is onto, these conditions
are equivalent to the existence of \(\zeta\in\mathbb F_2^3\) such that
\(A_i\binom{\zeta}{\chi_i}=0\) for \(i=1,2\).  Thus the cycles represented
by \(A_M\) are the symmetric differences of cycles of \(N_1\) and
\(N_2\) having the same intersection with \(T\); this is the cycle definition of
\(N_1\oplus_3N_2\).

This binary identification does not by itself show that the displayed real
matrix is totally unimodular.  To prove that, let \(H\) be a square
submatrix of \(A_M\), and let \(s\leq2\)
be the number of chosen rows among the first two rows of \(A_M\).
Most such minors factor immediately
along one side.  To see which do not, let \(a_i\) be the
number of chosen rows from \(Q_i\) and \(b_i\) the number of chosen columns.
If \(b_i<a_i\), those rows are supported on too few columns,
so \(\det H=0\).  We may therefore suppose \(b_i\geq a_i\), and write
\[
 \delta_i=b_i-a_i.
\]
For the selected rows and columns on side \(i\), write \(S_i\) and \(R_i\)
for the corresponding parts of \(P_i\) and \(Q_i\).  Since \(H\) is square,
\(\delta_1+\delta_2=s\).  Suppose, for example, that \(\delta_1=0\).
Expansion along the \(a_1=b_1\) chosen rows from \(Q_1\) has only one possible
column set, namely all chosen side-1 columns.  Hence
\[
 \det H=\pm\det R_1\det\binom{S_2}{R_2}.
\]
The two factors are minors of \(A_1\) and \(A_2\), respectively, so
\(\det H\in\{0,\pm1\}\).  The same argument applies when \(\delta_2=0\).
We are left with only one possibility:
\[
 s=2,\qquad \delta_1=\delta_2=1.
\]

For each side, define
\[
 c_i^j:=\det\begin{pmatrix}v_j&S_i\\0&R_i\end{pmatrix},
 \qquad j=1,2.
\]
Expand \(\det H\) along the \(a_1+1\) columns from side~1.  A nonzero term
must use all \(a_1\) rows from \(Q_1\) and exactly one of the two common
rows, leaving the other common row for side~2.  The two choices have
opposite Laplace signs.  Expanding \(c_i^j\) along \(v_j\) identifies their
signed sum as \(\pm\det(c_1,c_2)\), where
\(c_i=(c_i^1,c_i^2)^{\top}\).  The only possible obstruction would be
a pair of such vectors with determinant of absolute value greater than one.
Total unimodularity of the summands rules this out.  Adjoining each of the
triangle columns \(v_1,v_2\), and \(v_3=v_1+v_2\) in turn shows that
\[
 c_i^1,\qquad c_i^2,\qquad c_i^1+c_i^2
\]
are all minors of \(A_i\).  Each therefore belongs to \(\{0,\pm1\}\), which
leaves only
\[
 c_i\in\{0,\ \pm v_1,\ \pm v_2,\ \pm(v_1-v_2)\}.
\]
Any two vectors in this set have determinant \(0\) or \(\pm1\).  Hence
\(\det H\in\{0,\pm1\}\) also in the coupled case.  Thus every minor of
\(A_M\) lies in \(\{0,\pm1\}\), so \(A_M\) is totally unimodular.
The same argument covers an empty \(Q_i\)-block by the usual convention for
the empty minor.  This is the full-row form of the usual matrix three-sum;
for the corresponding compact-matrix rule and compatible signing convention,
see \citet[\S11.5, (11.5.5)--(11.5.8), pp.~262--264]{Truemper1998}.

It remains only to verify full row rank.  If a row vector
\((u,\alpha_1,\alpha_2)\) annihilates \(A_M\), then
\((u,\alpha_i)E_i=0\) for \(i=1,2\); the full row rank of \(E_i\) forces
\(u=\alpha_1=\alpha_2=0\).  Hence \(A_M\) has row rank
\(r_1+r_2-2=r(M)\).  Since its nonzero maximal minors are \(\pm1\),
the real matrix and its reduction modulo two have the same bases.
Thus \(A_M\) represents \(M\) over \(\mathbb R\).
\end{proof}

The lemma settles the combinatorial gluing.  For the optimisation argument,
we need only the part of each external Gram matrix that is visible from the
two common rows.  Fix the representations supplied by
\Cref{lem:compatible-three-sum-representation}.  For
\(x\in\R_{\geq0}^3\), put
\(\mathsf M(x)=\sum_{j=1}^3x_jv_jv_j^{\top}\).  Then
\begin{equation}
 \det\mathsf M(x)=e_2(x):=x_1x_2+x_1x_3+x_2x_3.
 \label{eq:triangle-cone-determinant}
\end{equation}

\subsection{A basis-polynomial factorisation}

The rows from \(Q_1\) and \(Q_2\) prevent the two side Gram matrices from
being added directly.  Schur complements factor out their contributions
to the determinant.  Total unimodularity forces each resulting
\(2\times2\) matrix to lie in the cone generated by
\(v_1v_1^{\top},v_2v_2^{\top},v_3v_3^{\top}\).

\begin{lemma}
\label{lem:three-sum-schur-response}
Suppose \(\lambda\geq0\) and its support spans \(M\).  Let
\(\lambda^{(i)}\) be its restriction to \(E(N_i)\setminus T\).
There are \(d_i>0\) and
\(x_i\in\R_{\geq0}^3\) such that
\begin{align}
 B_M(\lambda)&=d_1d_2e_2(x_1+x_2),
 \label{eq:three-sum-combined-response}\\
 B_{N_i}(\lambda^{(i)},w)&=d_i e_2(x_i+w)
 \label{eq:three-sum-completed-response}
\end{align}
for every nonnegative triangle-weight vector \(w\).
\end{lemma}

\begin{proof}
Separate the first two rows from the rows of \(Q_i\), and put
\(\Lambda_i=\diag(\lambda^{(i)})\).  The external Gram matrix on
side \(i\) is
\[
 \Gamma_i=
 \begin{pmatrix}P_i\\Q_i\end{pmatrix}
 \Lambda_i
 \begin{pmatrix}P_i^{\top}&Q_i^{\top}\end{pmatrix}
 =
 \begin{pmatrix}G_i&C_i\\ C_i^{\top}&D_i\end{pmatrix}.
\]
The positive-weight columns span \(M\), so their combined Gram matrix
\[
 \begin{pmatrix}
  G_1+G_2&C_1&C_2\\
  C_1^{\top}&D_1&0\\
  C_2^{\top}&0&D_2
 \end{pmatrix}
\]
is positive definite.  Hence its principal blocks satisfy
\(D_i\succ0\), with the empty-block convention when \(r_i=2\).  Define
\[
 d_i=\det D_i,
 \qquad
 H_i=G_i-C_iD_i^{-1}C_i^{\top}.
\]
The matrix \(Q_i\) represents the contraction \(N_i/T\), so Cauchy--Binet
gives \(d_i=B_{N_i/T}(\lambda^{(i)})\).  Restoring triangle weights \(w\)
adds \(\mathsf M(w)\) to \(G_i\).  Taking the Schur complement of \(D_i\)
therefore gives
\[
 B_{N_i}(\lambda^{(i)},w)=d_i\det\bigl(H_i+\mathsf M(w)\bigr).
\]

\filbreak
To determine \(H_i\), we compute the linear coefficients in \(w\) on both
sides, first by counting bases containing exactly one element of \(T\),
then by expanding the determinant.  Put
\(\lambda_S:=\prod_{e\in S}\lambda_e\), and let \(G_{ij}\) be the sum of
\(\lambda_S\) over independent \((r_i-1)\)-sets
\(S\subseteq E(N_i)\setminus T\) whose closure in \(N_i\) meets \(T\) in
\(\{t_j\}\).  In the representation, their spans meet the two-dimensional
span of \(T\) in \(\langle v_j\rangle\).  By total unimodularity, the
determinant on the columns of any basis is \(\pm1\).  Thus \(G_{ij}\) equals the
corresponding Cauchy--Binet contribution.

Every external \((r_i-1)\)-set that extends to a basis through \(T\) belongs
to one of these three families.  To see this, fix an ordering of such a set
\(S\), append \(t_k\) last, and put
\[
 \Delta_k=\det A_i[S,t_k].
\]
Then \((\Delta_2,-\Delta_1)\) spans the intersection with the span of \(T\).
Total unimodularity gives \(\Delta_k\in\{0,\pm1\}\), while
\(v_3=v_1+v_2\) gives \(\Delta_3=\Delta_1+\Delta_2\).  The sign pattern
\((1,-1)\) for the generator would force
\(\Delta_1=\Delta_2=\pm1\), and hence \(|\Delta_3|=2\), a contradiction.
The line is therefore one of
\(\langle v_1\rangle,\langle v_2\rangle,\langle v_3\rangle\).  Moreover,
\(S\cup\{t_k\}\) is a basis if and only if \(k\ne j\).

The respective coefficients of \(w_1,w_2,w_3\) are therefore
\[
 G_{i2}+G_{i3},\qquad G_{i1}+G_{i3},\qquad G_{i1}+G_{i2}.
\]
Expanding the \(2\times2\) determinant gives the same coefficients as
\[
 d_i(H_i)_{22},\qquad d_i(H_i)_{11},\qquad
 d_i\bigl((H_i)_{11}+(H_i)_{22}-2(H_i)_{12}\bigr).
\]
Equating the two expressions gives
\[
 d_iH_i=
 \begin{pmatrix}
  G_{i1}+G_{i3}&G_{i3}\\
  G_{i3}&G_{i2}+G_{i3}
 \end{pmatrix}.
\]
Hence \(H_i=\mathsf M(x_i)\), where \(x_{ij}=G_{ij}/d_i\geq0\).
Thus the three coordinates of \(x_i\) are weighted sums over these families,
divided by \(B_{N_i/T}(\lambda^{(i)})\).
Since \(H_i+\mathsf M(w)=\mathsf M(x_i+w)\), the determinant formula
\eqref{eq:triangle-cone-determinant} gives \eqref{eq:three-sum-completed-response}.
Taking the Schur complement of both \(D_i\)-blocks gives
\[
 B_M(\lambda)=d_1d_2\det(H_1+H_2)
              =d_1d_2e_2(x_1+x_2).
\]
This proves \eqref{eq:three-sum-combined-response}.
\end{proof}

The factorisation expresses each summand's contribution through \(d_i\) and
\(x_i\).  To compare the completed summands with the original sum, we relate
these quantities to the total weight on each side and bound the sum of the
coordinates of \(x_i\).

\filbreak
\begin{proposition}
\label{prop:three-sum-response-package}
Let \(a=r(N_1)\), \(b=r(N_2)\), and \(r=r(M)=a+b-2\).  Let \(\lambda\)
attain \(B_M^\star\), and put \(F=\{e:\lambda_e>0\}\).  Restriction to
\(F\) gives vectors \(x,y\in\R_{\geq0}^3\), factors
\(d_1,d_2>0\), and external side masses
\[
 m_i=\sum_{e\in F\cap(E(N_i)\setminus T)}\lambda_e.
\]
Put \(s=x+y\).  Then \(e_2(s)>0\).  If
\[
 q=\nabla\log e_2(s),
 \qquad
 \ell_1=x\cdot q,
 \qquad
 \ell_2=y\cdot q,
\]
then
\begin{equation}
 m_1+m_2=r,
 \qquad
 a-m_1=\ell_2,
 \qquad
 b-m_2=\ell_1,
 \label{eq:three-sum-mass-dictionary}
\end{equation}
and
\begin{equation}
 |x|\leq m_1,
 \qquad
 |y|\leq m_2.
 \label{eq:three-sum-trace-caps}
\end{equation}
Here \(|z|=z_1+z_2+z_3\).  In particular, \(|s|\leq r\).
\end{proposition}

\begin{proof}
Since \(B_M^\star>0\), the positive support \(F\) contains a basis.  Its
weighted Gram matrix is therefore positive definite, as are the
principal blocks \(D_1,D_2\).  Hence
\Cref{lem:three-sum-schur-response} applies to \(\lambda\), and its
factorisation together with \(B_M^\star>0\) gives \(e_2(s)>0\).

The KKT equations turn a uniform rescaling of one side into the required
mass identities.  They give
\(\partial_e\log B_M(\lambda)=1\) for \(e\in F\).  Scale the positive
side-\(i\) coordinates by \(\theta\).  Then
\[
 d_i(\theta)=\theta^{r_i-2}d_i,
 \qquad
 H_i(\theta)=\theta H_i;
\]
the block \(D_i\) has size \(r_i-2\).  Differentiation at
\(\theta=1\) therefore yields
\[
 m_i=r_i-2+
 \operatorname{tr}\!\left(H_i(H_1+H_2)^{-1}\right).
\]
The two traces sum to \(2\).  By
\eqref{eq:triangle-cone-determinant}, they equal \(x\cdot q\) and
\(y\cdot q\).  Every element of \(M\) lies on one of the two external
sides, so \(m_1+m_2=r\).  These identities prove
\eqref{eq:three-sum-mass-dictionary}.
To give the two summands total weights \(a\) and \(b\), we must add weights
\(\ell_2\) and \(\ell_1\), respectively, on the common triangle.

The mass identities alone do not control the size of \(s\).  That missing
bound comes from the possible first two coordinates of a column in the
totally unimodular representation.  Fix one summand and write
\[
 H=G-CD^{-1}C^{\top}=\mathsf M(z)
\]
for its Schur complement.  Let \(u_e\in\R^2\) be the first two entries of an external
column, and put \(m=\sum_e\lambda_e\), the total external mass on this side.
The entries of \(u_e\), and its determinants with the triangle columns, are
minors of the totally unimodular representation.  Thus, up to sign,
\[
 u_e\in\{0,v_1,v_2,v_3\}.
\]
A single quadratic form assigns value one to each of the three nonzero
directions.  Namely, take
\[
 P_0=
 \begin{pmatrix}1&-\tfrac12\\-\tfrac12&1\end{pmatrix}\succ0.
\]
If \(u_e\ne0\), then \(u_e^{\top}P_0u_e=1\).  Since
\(H\preceq G\),
\[
 |z|=\operatorname{tr}(P_0H)
 \leq\operatorname{tr}(P_0G)
 =\sum_e\lambda_eu_e^{\top}P_0u_e
 \leq m.
\]
We apply this estimate to both summands and obtain
\eqref{eq:three-sum-trace-caps}.
\end{proof}

\begin{remark}
In trace form, the bound \(|z|\leq m\) reads
\(\operatorname{tr}(P_0H)\leq m\).  We nevertheless retain the three
coordinates because optimal triangle completion depends on their individual
values.
\end{remark}

\subsection{Optimising weights on the common triangle}

The triangle is absent from \(M\), but it is available when either summand is
completed.  For a fixed coefficient vector \(z\) and a prescribed amount \(m\) of
triangle weight, the relevant question is how large the completed
determinant can be.  Accordingly, for \(z\geq0\) and \(m\geq0\), define
\begin{equation}
 \phi(z,m)=
 \max\{e_2(z+w):w\geq0,\ |w|=m\}.
 \label{eq:triangle-completion-definition}
\end{equation}
At the maximiser, the coordinates that receive additional weight have a
common final value, no larger than any coordinate left unchanged.

\begin{lemma}
\label{lem:triangle-water-filling}
Let \(z_{(1)}\geq z_{(2)}\geq z_{(3)}\) be the decreasing rearrangement of
\(z\), and put
\[
 L=|z|+m,
 \qquad
 t=\frac{L-z_{(1)}}2.
\]
Then
\[
\phi(z,m)=
\begin{cases}
 L^2/3
  & \text{if }L/3\geq z_{(1)}\\[1mm]
 2z_{(1)}t+t^2
  & \text{if }L/3<z_{(1)}\ \text{and}\ t\geq z_{(2)}\\[1mm]
 e_2\!\left(z_{(1)},z_{(2)},L-z_{(1)}-z_{(2)}\right)
  & \text{if }t<z_{(2)}
\end{cases}.
\]
\end{lemma}

\begin{proof}
For fixed \(|u|=L\),
\[
 e_2(u)=\frac12\left(L^2-\sum_ju_j^2\right).
\]
We therefore minimise \(\sum_ju_j^2\) subject to \(u_j\geq z_j\) and
\(\sum_ju_j=L\).  The feasible set is compact, and strict convexity gives a
unique minimiser.  If two coordinates above their lower bounds satisfy
\(u_p>u_q\), shifting a sufficiently small mass \(\epsilon\) from \(p\) to
\(q\) changes their contribution by
\(2\epsilon(u_q-u_p+\epsilon)<0\), a contradiction.  Thus all such
coordinates have a common value \(c\).  The same exchange shows that a
coordinate fixed at \(z_j\) must satisfy \(z_j\geq c\).  Consequently
\[
 u_j=\max\{z_j,c\},
 \qquad
 \sum_j\max\{z_j,c\}=L,
\]
where, if every coordinate is fixed, we may take \(c=z_{(3)}\).

If \(L/3\geq z_{(1)}\), then \(c=L/3\).  Otherwise the largest coordinate
is fixed; the other two equal \(t=(L-z_{(1)})/2\) if and only if
\(t\geq z_{(2)}\).  If \(t<z_{(2)}\), the first two coordinates are
fixed and the last is
\(L-z_{(1)}-z_{(2)}=z_{(3)}+m\geq z_{(3)}\).  Substitution in \(e_2\) gives
the three formulas.
\end{proof}

Optimal completion solves the problem on either side separately.  The split
\(s=x+y\), however, still ranges over a three-dimensional box.  Concavity
removes that freedom: the least completion product occurs when every
coordinate of \(s\) is assigned wholly to one side.

\begin{proposition}
\label{prop:separator-extreme-splits}
Let \(s\in\R_{\geq0}^3\) satisfy \(e_2(s)>0\), and put
\[
 q=\nabla\log e_2(s),
 \qquad
 p_j=s_jq_j.
\]
For \(I\subseteq\{1,2,3\}\), let \(s_I\) agree with \(s\) on \(I\) and be
zero on \(I^c\), and put \(p(I)=\sum_{j\in I}p_j\).  Then, for
every \(0\leq z\leq s\),
\begin{equation}
\begin{split}
 &\phi(z,2-z\cdot q)\phi(s-z,z\cdot q)\\
 &\qquad\geq
 \min_{I\subseteq\{1,2,3\}}
 \phi(s_I,2-p(I))\phi(s_{I^c},p(I)).
\end{split}
\label{eq:separator-extreme-splits}
\end{equation}
\end{proposition}

\begin{proof}
To reduce the minimum to the vertices, we first prove that
\[
 \mathcal C(z,m):=\sqrt{\phi(z,m)}
\]
is jointly concave in \(z\) and \(m\).
The map \(u\mapsto\sqrt{e_2(u)}=\det(\mathsf M(u))^{1/2}\) is concave on
\(\R_{\geq0}^3\).  We record the short proof.  For \(A\succ0\), \(B\succeq0\),
let \(\lambda_1,\lambda_2\geq0\) be the eigenvalues of
\(A^{-1/2}BA^{-1/2}\).  Then
\[
\begin{aligned}
 \det(A+B)^{1/2}
 &=\det(A)^{1/2}\sqrt{(1+\lambda_1)(1+\lambda_2)}\\
 &\geq\det(A)^{1/2}\bigl(1+\sqrt{\lambda_1\lambda_2}\bigr)\\
 &=\det(A)^{1/2}+\det(B)^{1/2},
\end{aligned}
\]
because
\[
 (1+\lambda_1)(1+\lambda_2)
 -\bigl(1+\sqrt{\lambda_1\lambda_2}\bigr)^2
 =(\sqrt{\lambda_1}-\sqrt{\lambda_2})^2.
\]
Replacing a singular \(A\succeq0\) by \(A+\varepsilon I\) and letting
\(\varepsilon\downarrow0\) gives the same inequality.  Thus
\(A\mapsto\det(A)^{1/2}\) is superadditive on
the positive-semidefinite \(2\times2\) cone; positive homogeneity converts
superadditivity into concavity.  This is the \(2\times2\) case of the
Minkowski determinant inequality; it also follows from
\citet[Theorem~4]{MarcusLopes1957}.  Since
\(u\mapsto\mathsf M(u)\) is linear and maps \(\R_{\geq0}^3\) into this cone,
the asserted concavity follows.

If \(w_i\) is optimal for \((z_i,m_i)\), then
\(\theta w_1+(1-\theta)w_2\) is feasible for the convex combination of
\((z_1,m_1)\) and \((z_2,m_2)\).  Hence \(\mathcal C\) is jointly concave on
\(\R_{\geq0}^3\times\R_{\geq0}\).

Apply this concavity to the split variable.  Here
\(q_j=(\sum_{k\ne j}s_k)/e_2(s)\geq0\), and Euler's identity gives
\(s\cdot q=2\).  Thus both completion masses below are nonnegative on
\(0\leq z\leq s\).  Put
\[
 C_1(z)=\mathcal C(z,2-z\cdot q),
 \qquad
 C_2(z)=\mathcal C(s-z,z\cdot q).
\]
Both functions are nonnegative and concave.  The geometric mean is concave
and coordinatewise increasing on \(\R_{\geq0}^2\).  Consequently
\[
 \Psi_s(z):=\sqrt{C_1(z)C_2(z)}
 =\bigl[\phi(z,2-z\cdot q)\phi(s-z,z\cdot q)\bigr]^{1/4}
\]
is concave on the split box.  Its minimum is therefore attained at a vertex
\(z=s_I\).  We raise the resulting inequality to the fourth power and obtain
\eqref{eq:separator-extreme-splits}.  At such a vertex, each coordinate of
\(s\) belongs entirely to one summand.
\end{proof}

The extreme-split reduction leaves the following estimate for the triangle
completion product.

\begin{theorem}
\label{thm:intrinsic-separator-contraction}
Let \(x,y\in\R_{\geq0}^3\), put \(s=x+y\), and suppose \(e_2(s)>0\).  Define
\[
 q=\nabla\log e_2(s),
 \qquad
 \ell_1=x\cdot q,
 \qquad
 \ell_2=y\cdot q,
 \qquad
 R_s=\max\{4,|s|\}.
\]
Then
\begin{equation}
 \phi(x,\ell_2)\phi(y,\ell_1)
 >\left(1+\frac1{R_s}\right)e_2(s).
 \label{eq:intrinsic-separator-contraction}
\end{equation}
\end{theorem}

\begin{proof}
Apply \Cref{prop:separator-extreme-splits} with \(z=x\).  It suffices to check
the eight choices of \(I\subseteq\{1,2,3\}\).  If a coordinate of \(s\) is
zero, different choices can give the same vertex.  The empty and full choices of \(I\)
satisfy the required inequality because their product is
\(\frac43e_2(s)\) and \(R_s\geq4\).
For any other \(I\), replacing it by \(I^c\) interchanges the two
factors because \(p(I^c)=2-p(I)\).  We may therefore assume \(|I|=1\),
and coordinate symmetry reduces it to
\[
 x=(A,0,0),
 \qquad
 y=(0,B,C),
 \qquad B\geq C\geq0.
\]
Since \(A+B+C=|s|\), \Cref{lem:pure-split-separator-gain} gives the strict
bound with
\(\max\{4,A+B+C\}=R_s\), namely
\eqref{eq:intrinsic-separator-contraction} for every nonempty proper subset
\(I\).  The minimum over all eight choices therefore satisfies the strict bound.
\end{proof}

The preceding theorem depends only on the coefficient vectors and their
total sum.  The trace bounds compare that sum with the rank of the three-sum and
give the product inequality needed for induction.

\begin{theorem}
\label{thm:rank-sensitive-three-sum-contraction}
Let \(N_1,N_2\) form a proper regular three-sum, with ranks \(a,b\geq3\),
and put \(r=a+b-2\).  Then
\begin{equation}
 B_M^\star
 <\frac r{r+1}B_{N_1}^\star B_{N_2}^\star.
 \label{eq:rank-sensitive-three-sum-contraction}
\end{equation}
\end{theorem}

\begin{proof}
Take the vectors supplied by \Cref{prop:three-sum-response-package} and put
\(s=x+y\).  Since \(r\geq4\) and \(|s|\leq r\), we have \(R_s\leq r\).
By \eqref{eq:three-sum-mass-dictionary}, completing side \(1\) with mass
\(\ell_2\) and side \(2\) with mass \(\ell_1\) gives feasible weightings on
the weight simplices of the two summands.  Hence
\begin{align*}
 B_{N_1}^\star B_{N_2}^\star
 &\geq d_1d_2\phi(x,\ell_2)\phi(y,\ell_1)\\
 &>\left(1+\frac1{R_s}\right)d_1d_2e_2(s)\\
 &\geq\left(1+\frac1r\right)B_M^\star.
\end{align*}
Rearranging proves \eqref{eq:rank-sensitive-three-sum-contraction}.
\end{proof}

\subsection{Rank-two factors}

The preceding inequality assumes that both factors have rank at least
three.  A rank-two factor has nonloops in only the three directions
of the common triangle, so the sum reduces to deletions and parallel
replacements in the other factor.  This case does not require the
quantitative saving above.

\begin{proposition}
\label{prop:rank-two-three-sum}
Let \(N_1,N_2\) be regular matroids whose ground sets meet in a coindependent triangle
\(T=\{t_1,t_2,t_3\}\), let \(r(N_1)=2\), and put
\(M=N_1\oplus_3N_2\).  For \(j=1,2,3\), let \(P_j\) consist of the elements
of \(E(N_1)\setminus T\) parallel to \(t_j\) in \(N_1\), and put
\[
 D=\{t_j:P_j=\varnothing\}.
\]
Deleting from \(M\) the loops of \(N_1\) gives the matroid obtained from
\(N_2\setminus D\) by replacing each \(t_j\in T\setminus D\) with parallel
copies indexed by \(P_j\).  Consequently,
\[
 B_M^\star=B_{N_2\setminus D}^\star\leq B_{N_2}^\star.
\]
\end{proposition}

\begin{proof}
In a binary representation of \(N_1\), the columns of \(T\) represent all
three points of \(\mathrm{PG}(1,2)\).  Thus every nonloop of
\(N_1\setminus T\) is parallel to a unique \(t_j\).
In the compatible representation, \(Q_1\) has no rows.  The external
columns of \(N_1\) are therefore zero for its loops and copies of the
corresponding triangle columns up to sign for its nonloops.  These columns
take the place of \(T\), while every external column of \(N_2\) is retained.

Loop coordinates do not occur in the basis polynomial.  If their total
weight in a feasible weighting is \(\eta<r(M)\), setting them to zero and
scaling every nonloop weight by \(r(M)/(r(M)-\eta)\) preserves feasibility
and multiplies \(B_M\) by
\((r(M)/(r(M)-\eta))^{r(M)}\geq1\).  If \(\eta=r(M)\), then \(B_M=0\).
Thus we may give the loops of \(N_1\) weight zero when computing \(B_M^\star\).

For a weighting \(\lambda\) on \(M\) with zero weight on those loops, set
\[
 \bar\lambda_{t_j}=\sum_{e\in P_j}\lambda_e
 \qquad(P_j\ne\varnothing),
\]
and leave all \(N_2\setminus T\) weights unchanged.  Parallel aggregation
gives
\[
 B_M(\lambda)=B_{N_2\setminus D}(\bar\lambda).
\]
Conversely, we realise every aggregate weighting by giving the loops of
\(N_1\) weight zero and concentrating each \(\bar\lambda_{t_j}\) on one
member of \(P_j\).  This proves equality of the
maxima.  Since \(D\subseteq T\) and \(T\) is coindependent,
\(r(N_2\setminus D)=r(N_2)=r(M)\).  Extending a weighting by zero on \(D\)
proves the final inequality.
\end{proof}

\subsection{The bound for three-sums}

We now combine the rank-two description with the inequality for factors of
rank at least three.  These two cases give the estimate used in the final
induction.

\begin{corollary}
\label{cor:three-sum-propagation}
Let \(M=N_1\oplus_3N_2\) be a proper regular three-sum.  If
\[
 B_{N_i}^\star\leq\frac{2^{r(N_i)}}{r(N_i)+1}
 \qquad(i=1,2),
\]
then
\[
 B_M^\star\leq\frac{2^{r(M)}}{r(M)+1}.
\]
If both factor ranks are at least three, the conclusion is strict.
\end{corollary}

\begin{proof}
Order the factor ranks as \(a\leq b\).  There are two cases.  If
\(a=2\),
\Cref{prop:rank-two-three-sum} gives
\[
 B_M^\star\leq B_{N_2}^\star
 \leq\frac{2^b}{b+1}
 =\frac{2^{r(M)}}{r(M)+1}.
\]
If \(a,b\geq3\), put \(r=a+b-2\).  By
\Cref{thm:rank-sensitive-three-sum-contraction},
\[
 B_M^\star
 <\frac r{r+1}\frac{2^{a+b}}{(a+1)(b+1)}
 \leq\frac{2^r}{r+1},
\]
because
\[
 (a+1)(b+1)-4r=(a-3)(b-3)\geq0.
\]
\end{proof}
\section{The regular case}
\label{sec:seymour-assembly}

The regular bound follows by induction from the estimates for the three
building blocks and the sum operations.  We use a recursive form of
Seymour's decomposition theorem
\cite[Theorem~14.3 and its proof]{Seymour1980}, which states that every regular
matroid is graphic, cographic, isomorphic to \(R_{10}\), or admits a
decomposition into regular matroids as a nontrivial direct sum, a proper
two-sum, or a proper three-sum.  In Seymour's construction, the terminal
pieces may be chosen isomorphic to minors of the original matroid.  A proper
two-sum
\(N_1\oplus_2N_2\) has a common element which is neither a loop nor a coloop
and satisfies \(\lvert E(N_i)\rvert\geq3\).  Proper three-sums use the
convention of \Cref{sec:regular-three-sum}, in particular
\(\lvert E(N_i)\rvert\geq7\).

To use the decomposition inductively, we must check that its factors are
smaller.  The three operations change rank and ground-set size as follows:
\[
\begin{array}{c|c|c}
 \text{sum}&r(M)&|E(M)|\\ \hline
 M_1\oplus M_2&r_1+r_2&|E(M_1)|+|E(M_2)|\\
 N_1\oplus_2N_2&r_1+r_2-1&|E(N_1)|+|E(N_2)|-2\\
 N_1\oplus_3N_2&r_1+r_2-2&|E(N_1)|+|E(N_2)|-6.
\end{array}
\]
Each factor has rank at most \(r(M)\).  In a proper two- or
three-sum it has strictly fewer elements.  In a direct sum, either its rank
is smaller or we remove a rank-zero summand.  Hence every factor used
below is smaller in the lexicographic order on \((r,|E|)\).

We now prove the regular case of the main theorem.  Its equality statement
will also be needed when we pass to binary matroids with no \(F_7\)-minor.

\begin{theorem}
\label{thm:regular-basis-mass}
Every regular matroid \(M\) of rank \(r\geq1\) satisfies
\[
 B_M^\star\leq\frac{2^r}{r+1}.
\]
A feasible weighting \(\lambda\) attains equality if and only if every loop
has weight zero, the simplification of \(M\) is isomorphic to \(M(K_{r+1})\), and
every parallel class has total weight \(2/(r+1)\).
\end{theorem}

\subsection{Proof of the inequality}

\begin{proof}[Proof of the inequality in
\Cref{thm:regular-basis-mass}]
We induct lexicographically on \((r(M),|E(M)|)\).  The terminal cases are
already covered, once we choose connected graph models.  A rank-zero factor
consists only of loops and has optimum
\(1\), so \Cref{lem:one-sum-scaling} removes it without changing the
optimum.  Every
graphic matroid has a connected graph model: delete isolated vertices and
identify one vertex from each remaining component.  This preserves the cycle
matroid and, after dualising, gives a connected graph model for every
cographic matroid.  Hence the graphic, cographic, and \(R_{10}\) cases follow
from
\Cref{prop:graphic-bound,thm:analytic-cographic-bound,prop:R10-bound}.

Suppose then that \(M\) is neither graphic nor cographic and is not isomorphic
to \(R_{10}\).  The recursive form of Seymour's theorem supplies a sum
decomposition.  Its factors are
lexicographically smaller, so the inductive bounds are available.  We use
\Cref{lem:one-sum-scaling} for direct sums,
\Cref{cor:two-sum-propagation} for two-sums, and
\Cref{cor:three-sum-propagation} for three-sums.  These exhaust the
possibilities and prove the inequality.
\end{proof}

\subsection{Equality}

To pass from the positive support to the whole matroid in the equality
case, we show that the support meets every parallel class.

\begin{lemma}
\label{lem:complete-graphic-support-rigidity}
Let \(M\) be a regular matroid of rank \(r\geq1\), and let \(A\subseteq E(M)\).
If the simplification of \(M|A\) is isomorphic to \(M(K_{r+1})\), then
every parallel class of \(M\) meets \(A\).
\end{lemma}

\begin{proof}
The parallel classes of \(M|A\) are the nonempty intersections of \(A\)
with the parallel classes of \(M\).  Thus \(A\) meets \(\binom{r+1}{2}\)
classes, one for each edge of \(K_{r+1}\).  The simplification of \(M\)
is a simple regular matroid of rank \(r\), hence binary with no
\(F_7\)-minor.  By \citet[Corollary~6.4]{Kung1993Extremal}, it has at
most \(\binom{r+1}{2}\) elements, so these are all the parallel classes of \(M\).
\end{proof}

\begin{proof}[Proof of the equality statement in
\Cref{thm:regular-basis-mass}]
Let \(\lambda\) be feasible and satisfy \(B_M(\lambda)=c_r\).  We begin by
removing features that the basis polynomial does not see separately.
No loop has positive weight: its derivative is zero, whereas every positive
coordinate has derivative \(B_M(\lambda)>0\) by
\Cref{lem:positive-support-reduction}.  Restrict to the positive support and
aggregate its parallel classes.  This gives a simple rank-\(r\) regular
matroid \(L\) with a positive weighting \(\mu\) such that
\[
 B_L(\mu)=c_r.
\]
The inequality already proved shows that \(\mu\) is a maximiser for \(L\).

We show that \(L\) is graphic by ruling out the other cases in Seymour's
theorem.  In a nontrivial direct sum both factors have positive rank, so
\Cref{lem:one-sum-scaling} gives a strict bound.  In a proper two-sum,
a rank-one factor has at least two external elements.  An external loop
remains a loop in the sum.  If there is no such loop, the external elements
are parallel in the factor and remain parallel in the sum.  Simplicity excludes both
possibilities.  Consequently both factor ranks are at least two, and
\Cref{cor:two-sum-propagation} is strict.

Likewise, a rank-two factor in a proper three-sum is impossible.  The common
triangle represents the three points of the binary projective line \(\mathrm{PG}(1,2)\).
The factor has at least four external elements; an external loop would give
a loop in the sum, while otherwise two external elements occupy the same
point and give a parallel pair.  Both contradict simplicity.  Thus both
factor ranks are at least three, and \Cref{cor:three-sum-propagation} is
strict.

The \(R_{10}\) case is strict by \Cref{prop:R10-bound}.  It remains to
consider \(L=M^*(G)\) with \(G\) connected and nonplanar.  A bridge of \(G\) would be a
loop of \(L\), so \(G\) is bridgeless.  As in the proof of
\Cref{thm:analytic-cographic-bound}, a graph loop or a decomposition at a
cut vertex gives a direct sum of positive-rank factors, for which the bound
is strict.  A degree-two vertex gives a parallel pair in \(L\), contrary to
simplicity.  If \(G\) has parallel edges, the rank-\((r-1)\) bound gives
\[
 B_L^\star\leq\frac43c_{r-1}
 =\frac{2(r+1)}{3r}\,c_r<c_r,
\]
where the strict inequality follows from \(r\geq4\) by
\Cref{lem:nonplanar-cycle-rank}.  In every remaining case,
the one-edge or cubic-star descent strictly increases the normalised optimum,
again contradicting equality.  Hence equality in the cographic case forces
planarity, and the matroid is graphic.  The graphic equality statement
\Cref{prop:graphic-bound} now gives
\[
 L\cong M(K_{r+1}),
 \qquad \mu_e=\frac2{r+1}\quad(e\in E(L)).
\]

By \Cref{lem:complete-graphic-support-rigidity}, every parallel class of
\(M\) meets the support of \(\lambda\).  Thus \(L\) is also a simplification
of \(M\), and restoring the zero-weight elements preserves the total
weight \(2/(r+1)\) of each class.

Conversely, parallel aggregation and \Cref{prop:graphic-bound} show that
every weighting listed in the theorem attains equality.
\end{proof}
\section{Binary matroids without a Fano minor}
\label{sec:fano-minor-free}

Seymour's theorem expresses every binary matroid with no \(F_7\)-minor as
a sequence of direct sums and two-sums of regular matroids and copies of
\(F_7^*\)
\citep{Seymour1981Multicommodity}; see also
\citet[Corollary~12.2.6]{Oxley2011}.  The regular factors are covered by
\Cref{thm:regular-basis-mass}, and the estimates in
\Cref{sec:regular-direct-sums,sec:regular-two-three-sums} control both sum
operations.  To prove \Cref{thm:introduction-basis-mass}, it remains to
bound the basis mass of \(F_7^*\) and check equality.

\subsection{The dual Fano matroid}

Every single-element contraction of \(F_7^*\) is graphic, so the graphic theorem reduces its
basis-mass estimate to a one-variable inequality.

\begin{lemma}
\label{lem:dual-fano-basis-mass}
The rank-four dual Fano matroid satisfies
\[
 B_{F_7^*}^\star
 =28\left(\frac47\right)^4
 =\frac{1024}{343}
 <\frac{16}{5}.
\]
\end{lemma}

\begin{proof}
Deleting any element from \(F_7\) gives \(M(K_4)\).  Dualising and using the
self-duality of \(M(K_4)\) shows that every contraction \(F_7^*/e\) is
isomorphic to \(M(K_4)\).  Let \(E=E(F_7^*)\), and let \(\lambda\) be a
nonnegative weighting of total mass four.  Counting each basis once for each
of its four elements, and then applying the homogeneous graphic bound to
each contraction, gives
\[
 4B_{F_7^*}(\lambda)
 =\sum_e\lambda_e
   B_{F_7^*/e}(\lambda|_{E\setminus\{e\}})
 \leq 2\sum_e\lambda_e
      \left(\frac{4-\lambda_e}{3}\right)^3.
\]
For each \(e\), the weights remaining after contraction have total mass
\(4-\lambda_e\); when this mass is zero, the corresponding summand is zero.
Hence
\[
 B_{F_7^*}(\lambda)
 \leq\frac1{54}\sum_e\lambda_e(4-\lambda_e)^3.
\]

It remains to bound a one-variable sum.  Put \(f(t)=t(4-t)^3\).  The tangent
to \(f\) at \(4/7\) lies above \(f\) on the whole real line, since
\[
 f\left(\frac47\right)
 +f'\left(\frac47\right)\left(t-\frac47\right)-f(t)
 =\left(t-\frac47\right)^2
  \left(\left(t-\frac{38}{7}\right)^2+\frac{284}{49}\right).
\]
There are seven elements and \(\sum_e\lambda_e=4\).  Summing the tangent
inequality therefore gives
\[
 B_{F_7^*}(\lambda)
 \leq\frac1{54}\sum_e f(\lambda_e)
 \leq\frac7{54}f\left(\frac47\right)
 =\frac{1024}{343}.
\]
Finally, \(F_7\) has \(\binom73-7=28\) bases, and its dual has the same
number.  Uniform weights attain the bound because
\(28(4/7)^4=1024/343\), and \(1024/343<16/5\).
\end{proof}

\subsection{Proof of the inequality}

\begin{proof}[Proof of the inequality in
\Cref{thm:introduction-basis-mass}]
We induct on the number of sum operations in the construction.  If no sum is
used, the matroid is regular or isomorphic to \(F_7^*\), and the result follows
from \Cref{thm:regular-basis-mass,lem:dual-fano-basis-mass}.  Otherwise,
consider the final operation.  Its factors have constructions with fewer sum
operations, so the inductive bounds are available.  Direct sums and two-sums
preserve the rank-dependent bound by
\Cref{lem:one-sum-scaling,cor:two-sum-propagation}.  A rank-zero factor
consists only of loops, has optimum one, and does not affect the direct-sum
step.  This completes the induction.
\end{proof}

\subsection{Equality}

We show that the simplification of a matroid attaining equality is regular.
The regular equality theorem then identifies the simplification and the
total weight of each parallel class.

\begin{proof}[Proof of the equality statement in
\Cref{thm:introduction-basis-mass}]
Let \(\lambda\) be a feasible weighting that attains the bound.  The
inequality just proved makes \(\lambda\) a maximiser, and
\Cref{lem:positive-support-reduction} shows that no loop has positive weight.

Delete the loops and replace each parallel class, including those of
total weight zero, by one representative carrying its total weight.
This produces a weighting \(\mu\) on the
simplification \(L\) of \(M\).  The matroid \(L\) is simple, binary,
has rank \(r\), and has no \(F_7\)-minor.  Parallel aggregation preserves both
total weight and basis mass, so \(\mu\) is feasible and
\[
 B_L(\mu)=\frac{2^r}{r+1}.
\]
Because the inequality already proved applies to \(L\), this identity implies
that \(B_L^\star=2^r/(r+1)\), and hence that \(\mu\) is a maximiser for \(L\).

We claim that \(L\) is regular.  Choose a Seymour construction of \(L\) with
as few sum operations as possible.  A direct sum with an empty factor can be
omitted, and a two-sum with a two-element factor merely relabels the common
element in the other factor.  Thus any final operation is a nontrivial direct
sum or a proper two-sum.  In the first case both factors have positive rank
because \(L\) has no loops, and \Cref{lem:one-sum-scaling} makes the bound
strict.  In the second case simplicity forces both factor ranks to be at least two.
A rank-one factor has at least two elements outside the common
element.  Any external loop remains a loop in the sum, while two external
nonloops remain parallel.  Both possibilities contradict simplicity.  The
strict part of \Cref{cor:two-sum-propagation} therefore excludes equality in
this case as well.

It follows that no sum is used, so \(L\) is regular or isomorphic to
\(F_7^*\); the latter possibility is strict by
\Cref{lem:dual-fano-basis-mass}.  The equality statement in
\Cref{thm:regular-basis-mass} now gives
\[
 L\cong M(K_{r+1}),
 \qquad
 \mu_e=\frac2{r+1}\quad(e\in E(L)).
\]
The elements of \(L\) are in bijection with the parallel classes of
\(M\).  Hence the simplification of \(M\) is isomorphic to \(M(K_{r+1})\), every
parallel class has total weight \(2/(r+1)\), and every loop has weight
zero.

Conversely, parallel aggregation and \Cref{prop:graphic-bound} show that
every weighting described in \Cref{thm:introduction-basis-mass} attains
equality.
\end{proof}

\begin{proof}[Proof of \Cref{cor:largest-minor-closed-class}]
Uniform weights already violate the bound on \(U_{2,4}\) and \(F_7\).
The six bases of \(U_{2,4}\) give \(6(1/2)^2=3/2>4/3\), and the
\(28\) bases of \(F_7\) give \(28(3/7)^3=108/49>2\).
Any minor-closed class satisfying the stated inequality must therefore
exclude both \(U_{2,4}\) and \(F_7\).  Since the class is minor-closed,
Tutte's characterisation shows that every member is binary
\cite[Theorem~6.5.4]{Oxley2011}, and the exclusion of \(F_7\) shows that no
member has an \(F_7\)-minor.  Thus every such class is contained in the class
of binary matroids with no \(F_7\)-minor.  Conversely, the main theorem proves
the inequality for every positive-rank matroid in this class, while a
rank-zero matroid has \(B_M^\star=1=2^0/(0+1)\).
\end{proof}

\subsection{Consequences for the number of bases}

Evaluating the basis polynomial at uniform weights turns the main theorem
into a sharp basis-count bound.  Its equality case identifies the matroids
whose uniform weighting attains that bound.

\begin{corollary}
\label{cor:fano-free-basis-count}
Let \(M\) be an \(n\)-element rank-\(r\) binary matroid with no
\(F_7\)-minor, where \(n\geq r\geq1\), and write
\(b(M)=|\mathcal B(M)|\).  Then
\[
 b(M)\leq
 \frac{2^r}{r+1}\left(\frac nr\right)^r.
\]
Equality holds if and only if \(\binom{r+1}{2}\) divides \(n\) and \(M\) is
obtained from \(M(K_{r+1})\) by replacing every edge with a parallel class
of size \(n/\binom{r+1}{2}\).
\end{corollary}

\begin{proof}
Give every element weight \(r/n\).  Then
\[
 B_M(\lambda)=b(M)\left(\frac rn\right)^r,
\]
so the inequality follows from \Cref{thm:introduction-basis-mass}.  If
equality holds, this uniform weighting attains equality in the main theorem.
The equality statement therefore shows that \(M\) has no loops and has
\(\binom{r+1}{2}\) parallel classes, each of total weight
\(2/(r+1)\).  Since every individual weight is \(r/n\), each class has size
\[
 \frac{2n}{r(r+1)}
 =\frac{n}{\binom{r+1}{2}}.
\]
This proves the necessary form and the divisibility condition.  Conversely,
the same calculation shows that the matroid with these class
sizes attains the bound.
\end{proof}

For simple matroids, the corollary combines with the extremal size
bound.  By \citet[Corollary~6.4]{Kung1993Extremal}, a simple rank-\(r\) binary
matroid with no \(F_7\)-minor has at most \(\binom{r+1}{2}\) elements.  Hence
\Cref{cor:fano-free-basis-count} gives
\[
 b(M)\leq\frac{2^r}{r+1}
 \left(\frac{\binom{r+1}{2}}r\right)^r
 =(r+1)^{r-1}.
\]
Equality forces
\(|E(M)|=\binom{r+1}{2}\) and equality in the corollary, and hence
\(M\cong M(K_{r+1})\).  The converse follows from Cayley's formula.

The same construction also shows that divisibility does not affect the
leading asymptotic term.

\begin{corollary}
\label{cor:fano-free-limiting-density}
Fix \(r\geq1\).  For \(n\geq r\), let \(b_r(n)\) be the largest number of
bases of an \(n\)-element rank-\(r\) binary matroid with no \(F_7\)-minor.
Then
\[
 \lim_{n\to\infty}\frac{b_r(n)}{\binom nr}
 =\frac{2^r r!}{r^r(r+1)}.
\]
\end{corollary}

\begin{proof}
Put \(m=\binom{r+1}{2}\).  For each \(n\geq m\), replace the edges of
\(K_{r+1}\) by parallel classes whose sizes sum to \(n\) and differ by at
most one.  The cycle matroid \(M_n\) of the resulting graph belongs to the
stated class.  Each spanning tree of \(K_{r+1}\) lifts by choosing one copy
of each of its \(r\) edges.  Each edge has \(n/m+O(1)\) copies, so Cayley's
formula gives
\[
 b(M_n)=(r+1)^{r-1}\left(\frac nm\right)^r+O(n^{r-1})
       =\frac{2^r}{r^r(r+1)}n^r+O(n^{r-1}).
\]
On the other hand, \Cref{cor:fano-free-basis-count} gives
\[
 b(M_n)\leq b_r(n)\leq\frac{2^r}{r^r(r+1)}n^r.
\]
Since \(\binom nr=n^r/r!+O(n^{r-1})\) for fixed \(r\), the limit follows.
\end{proof}

For general nondivisible \(n\), these estimates leave the exact maximum
undetermined.  In the discussion following Problem~7.7, \citet{vanDerPolWalshWigal2025}
consider whether a graphic matroid attains \(b_r(n)\) for all sufficiently
large \(n\).

\filbreak
\begin{corollary}
\label{cor:fano-free-eventual-graphic}
Fix \(r\geq1\).  For all sufficiently large \(n\), every \(n\)-element
rank-\(r\) binary matroid with no \(F_7\)-minor and \(b_r(n)\) bases
has simplification isomorphic to \(M(K_{r+1})\) and hence is graphic.
\end{corollary}

\begin{proof}
Let \(M\) attain \(b_r(n)\), and let \(S\) be its simplification.  The loops
form a rank-zero direct summand, so deleting them preserves the optimum by
\Cref{lem:one-sum-scaling}.  Merging each parallel class also preserves the
optimum by \Cref{lem:parallel-merging}, since every weighting of \(S\) lifts
by placing each class total on one representative.  Evaluation at uniform
weights on \(M\) therefore gives
\[
 \left(\frac rn\right)^r b_r(n)\leq B_M^\star=B_S^\star.
\]
By \Cref{cor:fano-free-limiting-density}, the left-hand side tends to \(c_r\).

Every simple rank-\(r\) binary matroid is represented by a subset of
\(\mathbb F_2^r\setminus\{0\}\), so there are only finitely many possibilities
for \(S\) up to isomorphism.  As a restriction of \(M\), \(S\) has no
\(F_7\)-minor.  Hence \Cref{thm:introduction-basis-mass} gives
\(B_S^\star<c_r\) for every such \(S\) with \(S\not\cong M(K_{r+1})\).  For all
sufficiently large \(n\), \((r/n)^r b_r(n)\) exceeds every one of these
smaller optima, forcing \(S\cong M(K_{r+1})\).  Replacing each edge of
\(K_{r+1}\) by its parallel class and adding the loops of \(M\) gives a
graph with cycle matroid \(M\).
\end{proof}

\clearpage
\phantomsection
\addcontentsline{toc}{section}{References}


\clearpage
\appendix
\section{The scalar completion inequality}
\label{app:three-sum-pure-split}

The proof of \Cref{thm:intrinsic-separator-contraction} has reduced the
completion problem to a one-versus-two split: one summand carries \((A,0,0)\), and
the other carries \((0,B,C)\).  The remaining question is how large the
completion product is relative to \(e_2(A,B,C)\).  The boundary
\(C=0\) already shows why the estimate must retain the total scale: when
\(A=B=T\), the gain is asymptotic to \(1/(A+B)\) as \(T\to\infty\).
We therefore determine which
lower bounds are attained in the completions and settle the regions where the gain is
uniformly larger than required.  The remaining calculation isolates the
behaviour near this boundary and compares it with \(A+B+C\).

\begin{lemma}
\label{lem:pure-split-separator-gain}
Let \(A,B,C\geq0\), with \(B\geq C\) and \(e_2(A,B,C)>0\).  Put
\(\tau=A+B+C\), \(x=(A,0,0)\), \(y=(0,B,C)\), and
\(q=\nabla\log e_2(x+y)\).  If
\[
 \nu=x\cdot q,
 \qquad
 \mu=y\cdot q=2-\nu,
\]
then
\begin{equation}
 \phi(x,\mu)\phi(y,\nu)
 >\left(1+\frac1{\max\{4,\tau\}}\right)e_2(x+y).
 \label{eq:pure-split-separator-gain}
\end{equation}
\end{lemma}

\begin{proof}
The pair \((B,C)\) contributes both its total size and the interaction between
its coordinates.  Separate them by setting
\[
 S=B+C,
 \qquad
 h=\frac{BC}{S},
 \qquad
 \kappa=A+h,
 \qquad
 E=e_2(A,B,C)=S\kappa.
\]
Thus \(S,\kappa>0\), and
\[
 \nu=\frac A\kappa,
 \qquad
 \delta=1-\nu=\frac h\kappa,
 \qquad
 \mu=1+\delta.
\]
We measure the desired gain by the normalised product
\(
 \mathcal Q=\phi(x,\mu)\phi(y,\nu)/E.
\)
Also write \(t_+=\max\{t,0\}\).  Since \(\tau>0\), the desired estimate is
equivalent to the single normalised target
\begin{equation}
 \mathcal Q-1>
 \min\left\{\frac14,\frac1\tau\right\}.
 \label{eq:pure-split-normalized-target}
\end{equation}
Thus there are two possible ways to finish: the uniform estimate
\(\mathcal Q>5/4\), which immediately gives
\eqref{eq:pure-split-normalized-target}, or, if that fails, an estimate
retaining the scale \(\tau\).  Before choosing between them, we must know
which lower bounds are attained in the completions.  The water-filling formula
\Cref{lem:triangle-water-filling} gives
\begin{align*}
 \phi(x,\mu)
 &=A\mu+\frac{\mu^2}{4}+\frac{(\mu-2A)_+^2}{12},\\
 \phi(y,\nu)
 &=BC+\nu S+\frac{(\nu-C)_+^2}{4}
   +\frac{(C+\nu-2B)_+^2}{12}.
\end{align*}
The three positive parts change formula at
\[
 A=\frac\mu2,
 \qquad C=\nu,
 \qquad C+\nu=2B.
\]
In the regions \(A\leq\mu/2\) and \(C+\nu\geq2B\), direct feasible
completions already give \(\mathcal Q\geq4/3\).  We treat these two regions
first; afterward, the argument itself will determine on which side of the
middle transition we lie.

If \(A\leq\mu/2\), use \((\nu,0,0)\) as a feasible completion for
\(\phi(y,\nu)\).  Together with the exact value of \(\phi(x,\mu)\), this
gives
\begin{align*}
 3E\left(\mathcal Q-\frac43\right)
 &\geq S(A+\mu)^2(h+\nu)-4E\\
 &=S\nu(\kappa-1)^2
   \bigl(\mu^2+\nu\delta\kappa\bigr)\geq0.
\end{align*}
For the second region, suppose that \(C+\nu\geq2B\).  Use
\((0,\mu B/S,\mu C/S)\) as a feasible completion for \(\phi(x,\mu)\).
Together with the exact value of \(\phi(y,\nu)\), this gives
\begin{align*}
 3E\left(\mathcal Q-\frac43\right)
 &\geq (S+\nu)^2
   \left(A\mu+\frac{\mu^2h}{S}\right)-4E\\
 &=\frac{A(S-\mu)^2}{S}
   (S\mu+\nu\delta)\geq0.
\end{align*}
In either situation, \(\mathcal Q\geq4/3>5/4\), so
\eqref{eq:pure-split-normalized-target} holds.  The first situation includes
\(A=0\).  We may therefore assume
\[
 A>\frac\mu2,
 \qquad
 C+\nu<2B.
\]

Under these two inequalities the first and third positive parts vanish.  If
we temporarily discard the remaining nonnegative correction, the product
reduces to
\[
 R_0=\mu\left(1+\frac{\mu}{4A}\right)(A\delta+\nu^2).
\]
Substitution in the two completion formulas gives
\[
 \mathcal Q
 =R_0+\mu\left(1+\frac{\mu}{4A}\right)
       \frac{\nu(\nu-C)_+^2}{4S}.
\]
The uniform bound will therefore follow whenever \(R_0\geq4/3\).  The next
two identities isolate the only sign that can prevent this:
\begin{align*}
 R_0-\frac{3\mu}{4}(2\nu^2+\mu\delta)
 &=\frac{\mu}{4A}(2A-\mu)(2A\delta-\nu^2),\\
 \frac{3\mu}{4}(2\nu^2+\mu\delta)
 &=\frac43+
   \frac{(3\delta-1)^2(3\delta+2)}{12}.
\end{align*}
Since \(2A-\mu>0\), the only undecided sign is that of
\(2A\delta-\nu^2\).  Here \(A>0\), so \(\nu>0\).  Set
\[
 \sigma=\frac{2A\delta}{\nu^2}=\frac{2h}{\nu},
\]
where the second equality follows from \(A\delta=\nu h\).

Then \(2A\delta-\nu^2=\nu^2(\sigma-1)\), so its sign is the sign of
\(\sigma-1\).
Unlike \(\delta\), the parameter \(\sigma\) changes under a common rescaling
of \((A,B,C)\).  If \(\sigma\geq1\), this quantity is nonnegative.
The two identities above and the nonnegative second term in \(\mathcal Q\)
give \(\mathcal Q\geq R_0\geq4/3\).  Hence
\eqref{eq:pure-split-normalized-target} holds for \(\sigma\geq1\).

Only the range \(0\leq\sigma<1\) remains.  We first consider the endpoint
\(\sigma=0\).  Since \(S>0\) and \(B\geq C\geq0\), we have \(\sigma=0\)
if and only if \(C=0\).  In that case \(\delta=0\),
\(\mu=\nu=1\), and
\[
 \mathcal Q-1
 =\frac1{4A}+\frac1{4B}+\frac1{16AB}
 >\frac14\left(\frac1A+\frac1B\right)
 \geq\frac1{A+B}
 \geq\min\left\{\frac14,\frac1\tau\right\}.
\]
Thus \eqref{eq:pure-split-normalized-target} holds when \(\sigma=0\).  This
endpoint also explains why the statement must retain the scale \(\tau\): if
\(A=B=T\), then \(\tau=2T\) and the gain is
\[
 \mathcal Q-1=\frac1\tau+\frac1{4\tau^2},
\]
which is asymptotic to \(1/\tau\) as \(T\to\infty\).  A uniform gain cannot control
this boundary.

We may henceforth assume \(C>0\), so \(h,\delta>0\) and \(0<\sigma<1\).
Since \(B\geq C\), we have \(B/S\geq1/2\), and hence
\(C\leq2BC/S=2h=\sigma\nu<\nu\).  The remaining positive part
\((\nu-C)_+\) therefore equals \(\nu-C\).

The three positive parts are now determined.  Set \(\theta=C/S\), which
measures the balance between \(B\) and \(C\).
Since \(C>0\) and \(B\geq C\), we have \(0<\theta\leq1/2\).  Equality
would give \(B=C\), and then \(C+\nu<2B\) would imply \(\nu<C\),
contrary to \(C<\nu\).  Thus \(0<\theta<1/2\).  Moreover,
we already know that \(0<\sigma<1\).  Since
\(h=S\theta(1-\theta)=A\delta/\nu\), solving the definitions for \(A\)
and \(S\) gives
\begin{align*}
 A&=\frac{\sigma(1-\delta)^2}{2\delta},
 &S&=\frac{\sigma(1-\delta)}{2\theta(1-\theta)}.
\end{align*}
The standing assumptions \(A>\mu/2\) and \(C+\nu<2B\) become,
respectively,
\begin{align*}
 \sigma&>\frac{\delta(1+\delta)}{(1-\delta)^2},
 &\sigma&>\frac{2\theta(1-\theta)}{2-3\theta}.
\end{align*}
The first inequality, together with \(\sigma<1\), gives
\(\delta<1/3\) and \(\sigma>\delta\), while the second gives
\(\sigma>\theta\).  We have therefore shown that
\[
 0<\delta<\frac13,
 \qquad
 0<\theta<\frac12,
 \qquad
 \max\{\delta,\theta\}<\sigma<1.
\]

The parameters \(\delta\), \(\theta\), and \(\sigma\) determine \(A\), \(S\),
and \(\tau=A+S\).  Define
\[
 X=\delta\bigl[(1+\delta)^2
       -2\sigma(1+\delta-\delta^2)\bigr],
 \qquad
 Y=\frac{\theta}{1-\theta}
       \bigl[2(1-\theta)-\sigma\bigr]^2.
\]
Since \(\theta>0\) and \(2(1-\theta)>1>\sigma\), we have \(Y>0\).
Direct substitution gives
\[
 \mathcal Q=
 \frac{(2\sigma+X)\,[4\sigma(\sigma+2)+Y]}{16\sigma^2}.
\]

The lower bound for \(2\sigma+X\) changes at \(\sigma=1/2\).  To see why, put
\[
 \delta_*=\frac{2\sigma-1}{2\sigma+1}.
\]
Factoring the difference from \(16\sigma^2/(2\sigma+1)^2\) gives
\[
 2\sigma+X
 =\frac{16\sigma^2}{(2\sigma+1)^2}
  +\bigl((2\sigma+1)\delta+2\sigma\bigr)
       (\delta-\delta_*)^2.
\]
Since \(\delta_*\) is nonnegative exactly when \(\sigma\geq1/2\), we split
at this threshold, where the resulting lower bound for \(\mathcal Q\) is \(5/4\).

Suppose first that \(1/2\leq\sigma<1\).  Since \(Y>0\), this factorisation
gives
\begin{align*}
 \mathcal Q
 &>\frac{4\sigma(\sigma+2)}{(2\sigma+1)^2}
 =\frac54+
   \frac{(2\sigma-1)(5-2\sigma)}{4(2\sigma+1)^2}
 \geq\frac54.
\end{align*}
Thus \(\mathcal Q>5/4\) whenever \(1/2\leq\sigma<1\), and
\eqref{eq:pure-split-normalized-target} follows.

We have reached the only case not covered by a uniform bound:
\(0<\sigma<1/2\).  Here the total scale \(\tau\) must be retained.  Put
\(\eta=1-2\sigma>0\).  It is enough to prove
\[
 \mathcal Q-1>\frac\sigma2+\frac\eta\tau.
\]
The right-hand side is the convex combination
\[
 \frac\sigma2+\frac\eta\tau
 =2\sigma\,\frac14+(1-2\sigma)\frac1\tau
 \geq\min\left\{\frac14,\frac1\tau\right\}.
\]
The excess beyond \(\sigma/2\) separates into contributions from the two
completion factors:
\[
 X-\delta^2(1+2\delta)
 =\delta\eta(1+\delta-\delta^2)>0,
\]
so expansion gives
\[
 \mathcal Q-1
 =\frac\sigma2+\frac{4X+Y}{8\sigma}
  +\frac{X(4\sigma^2+Y)}{16\sigma^2}
 >\frac\sigma2+\frac{4X+Y}{8\sigma}.
\]
It is therefore enough to prove
\((4X+Y)/(8\sigma)>\eta/\tau\).

We must compare one contribution associated with \(A\) and one associated
with \(S\) against \(1/(A+S)\).  This suggests choosing positive quantities
whose reciprocals recover \(A+S\).  Set
\[
 c_A=\frac{\delta}{(1-\delta)^2}=\frac\sigma{2A},
 \qquad
 c_S=\frac{\theta(1-\theta)}{1-\delta}=\frac\sigma{2S}.
\]
Their reciprocals satisfy
\[
 c_A^{-1}+c_S^{-1}=\frac{2\tau}{\sigma}.
\]
Together with
\((c_A+c_S)(c_A^{-1}+c_S^{-1})\geq4\), this gives
\[
 c_A+c_S\geq\frac{2\sigma}{\tau}.
\]
Consequently, the proof has been reduced to one estimate from each side:
\[
 X>\eta c_A,
 \qquad
 Y>4\eta c_S.
\]

For the singleton side, define the affine function
\[
 f_\delta(s)=\frac{\delta[(1+\delta)^2
       -2s(1+\delta-\delta^2)]}{c_A}-(1-2s).
\]
On \(0<\delta<1/3\), the identities
\[
 f_\delta(\sigma)-f_\delta(\delta)
 =2\delta\bigl[1+\delta(1-\delta)(2-\delta)\bigr]
   (\sigma-\delta)>0,
 \qquad
 f_\delta(\delta)
 =\delta^3\bigl[(1-\delta)(4-2\delta)+\delta\bigr]>0
\]
imply \(f_\delta(\sigma)>0\).  Since
\(X/c_A-\eta=f_\delta(\sigma)\), we obtain \(X>\eta c_A\).

For the pair side, the required estimate follows from
\begin{align*}
 \frac{Y}{c_S}
 &=(1-\delta)\left(2-\frac{\sigma}{1-\theta}\right)^2\\
 &>(1-\sigma)
   \left(2-\frac\sigma{1-\sigma}\right)^2
 =4\eta+\frac{\sigma^2}{1-\sigma}>4\eta.
\end{align*}
The strict inequality in the second line follows from
\(1-\delta>1-\sigma\) and
\[
 2-\frac{\sigma}{1-\theta}>
 2-\frac{\sigma}{1-\sigma}>0.
\]
Thus \(Y>4\eta c_S\).  Combining this with \(X>\eta c_A\) and the bound on
\(c_A+c_S\) gives
\[
 \frac{4X+Y}{8\sigma}
 >\frac{4\eta(c_A+c_S)}{8\sigma}
 \geq\frac\eta\tau.
\]
Hence \eqref{eq:pure-split-normalized-target} also holds when
\(0<\sigma<1/2\), completing the proof.
\end{proof}

\end{document}